\documentclass[a4paper,11pt]{amsart}
\usepackage{enumerate}
\usepackage{amsmath}
\usepackage{amssymb}
\usepackage{graphicx}
\usepackage{setspace}
\usepackage{geometry}

\newtheorem{theorem}{Theorem}[section]
\newtheorem{corollary}[theorem]{Corollary}
\newtheorem{lemma}[theorem]{Lemma}
\newtheorem{proposition}[theorem]{Proposition}
\theoremstyle{definition}
\newtheorem{definition}[theorem]{Definition}
\newtheorem{example}[theorem]{Example}
\newtheorem{remark}[theorem]{Remark}

\numberwithin{equation}{section}

\title[On $d$--$\sigma$--stability in random metric spaces and its applications ]{On $d$--$\sigma$--stability in random metric spaces and its applications$^*$}

\author[T. X. Guo]{Tiexin Guo$^1$}

\address[T. X. Guo]{School of Mathematics and Statistics, Central South University,
 Changsha 410083, China}
\email{{\tt tiexinguo@csu.edu.cn}}
\thanks{$^*$This paper is supported by the NNSF of China (No. 11571369).}
\thanks{$^1$ Corresponding author}

\author[E. X. Zhang]{Erxin Zhang}
\address[E. X. Zhang]{School of Mathematics and Statistics, Central South University,
 Changsha 410083, China}
\email{\tt zhangerxin6666@163.com}

\author[Y. C. Wang]{Yachao Wang}
\address[Y. C. Wang]{School of Mathematics and Statistics, Central South University,
 Changsha 410083, China}
\email{\tt wychao@csu.edu.cn}

\author[B. X. Yang]{Bixuan Yang}
\address[B. X. Yang]{School of Mathematics and Statistics, Central South University,
 Changsha 410083, China}
\email{\tt bixuanyang@126.com}



\begin{document}

\begin{abstract}
In 2010, the first author of this paper introduced the notion of $\sigma$--stability for a nonempty subset of an $L^0(\mathcal{F},K)$--module in [T.X. Guo, Relations between some basic results derived from two kinds of topologies for a random locally convex module, J. Funct. Anal. 258(2010), 3024--3047], this kind of $\sigma$--stability is purely algebraic and leads to a series of deep developments of random normed modules and random locally convex modules. Motivated by this, A. Jamneshan, M. Kupper and J. M. Zapata recently introduced another kind of $\sigma$--stability for a nonempty subset of a random metric space $(E,d)$, called $d$--$\sigma$--stability since it depends on the random metric $d$. $d$--$\sigma$--stability coincides with the previous $\sigma$--stability in the case of random normed modules, which motivates us in this paper to generalize the precise form of Ekeland's variational principle from a complete random normed module to a complete $d$--$\sigma$--stable random metric space. Besides, this paper also utilize $d$--$\sigma$--stability to generalize Nadler's fixed point theorem for a multivalued contraction mapping from a complete metric space to a complete random metric space. To our surprise, our simple fixed point theorem, however, can derive the known basic fixed point theorems of contraction type for both random operators and $\sigma$--stable mappings on a complete random normed module. A lot of examples shows the study of random metric spaces is more complicated than that of random normed modules.
\end{abstract}

\maketitle


\section{Introduction}\label{section1}
Let $(\Omega, \mathcal{F}, P)$ be a probability space, $K$ the scalar field $R$ of real numbers or $C$ of complex numbers and $L^0(\mathcal{F}, K)$ the algebra of equivalence classes of $K$--valued random variables on $(\Omega, \mathcal{F}, P)$. For a left module $E$ over $L^0(\mathcal{F}, K)$ $($ briefly, an $L^0(\mathcal{F}, K)$-module $)$ and a nonempty subset $G$ of $E$, $G$ is said to be $\sigma$--stable $($ originally called `` having the countable concatenation property ''in \cite{Guo10}, see \cite[Def.3.1]{Guo10}$)$ if there exists some $x \in G$ such that $\tilde{I}_{A_n} \cdot x = \tilde{I}_{A_n} \cdot x_n$ for each $n \in N$, for each sequence $\{ x_n : n \in N \}$ in $G$ and each countable partition $\{ A_n : n \in N \}$ of $\Omega$ to $\mathcal{F}$. The initial aim of introducing $\sigma$--stability in \cite{Guo10} is to establish the inherent connections between some basic results derived from two kinds of topologies for a random normed module or $($ more generally $)$ a random locally convex module, a series of subsequent developments have attested crucial roles played by the notion of $\sigma$--stability, see, e.g. \cite{CKV15,DJKK16,DKKS13,FM14a,FM14b,Guo13,GY12,GZWG18,GZWYYZ17,GZZ14,WG15,Z17}. Clearly, $\sigma$--stability only depends on the $L^0(\mathcal{F}, K)$--module structure and thus is purely algebraic.

It is well known that Banach's contraction mapping principle \cite{B22} and Ekeland's variational principle \cite{BAX66} on a complete metric space are two of the most powerful tools in functional analysis. Random metric spaces $($ briefly, $RM$ spaces $)$ are a random generalization of ordinary metric spaces. Roughly speaking, an $RM$ space with base $(\Omega, \mathcal{F}, P)$ is an ordered pair $(E,d)$ such that the random metric $d : E \times E \to L^0_+(\mathcal{F}) := \{ \xi \in L^0(\mathcal{F},R) : \xi \geq 0 \}$ satisfies the axioms similar to those satisfied by an ordinary metric, see Section \ref{section2} of this paper. The two principles stated above are already generalized to complete $RM$ spaces \cite{Guo99,GY12}. But when we recently used the result of \cite{Guo99} to study the existence and uniqueness of a class of backward stochastic equations in \cite{GZWG18}, where we were forced to consider a kind of random contraction mapping on a $\sigma$--stable subset of a complete random normed module $($ briefly, $RN$ module $)$ since the random iteration of such a mapping heavily depends on $\sigma$--stability. Likewise, when Guo and Yang \cite{GY12} attempted to establish the precise form of Ekeland's variational principle, they could only give the corresponding result for $\sigma$--stable complete $RN$ modules since $\sigma$--stability was essential. $RN$ modules are a class of important $RM$ spaces, we naturally would like to generalize some basic results involved in \cite{GY12,GZWG18} to general complete $RM$ spaces $($ namely not just complete $RN$ modules $)$, but the problem is that the notion of $\sigma$--stability introduced in \cite{Guo10} is not applicable to general $RM$ spaces since they are not necessarily $L^0(\mathcal{F}, K)$--modules in general. Recently, A. Jamneshan, M. Kupper and J. M. Zapata introduced in \cite{JKZ18} another kind of $\sigma$--stability for a nonempty subset of an $RM$ space, called $d$--$\sigma$--stability since it only depends on the random metric $d$. It is not difficult to see that $d$--$\sigma$--stability coincides with $\sigma$--stability in the case of $RN$ modules. With the notion of $d$--$\sigma$--stability, we are able to generalize some basic results in \cite{GY12,GZWG18} to a $d$--$\sigma$--stable complete $RM$ space. Besides, we are also able to generalize Nadler's elegant fixed point theorem \cite{N69} for multivalued contraction mappings from a complete metric space to a complete $RM$ space, to our surprise, our result can derive the well--known random fixed point theorems such as those given by O.Han\v{s} \cite{H61} and by S.Iton \cite{I77} as well as the fixed theorem of contraction type \cite{GZWG18} for $\sigma$--stable mappings on complete $RN$ modules.

Now, random functional analysis $($ according to Guo \cite{Guo93,GZWYYZ17}, which can be aptly defined as functional analysis based on $RM$ spaces, $RN$ modules, random inner product modules $($ $RIP$ modules $)$ and random locally convex modules $)$ has undergone a systematic and deep development. Some important advances in random functional analysis can be briefly surveyed as follows in order for the scholars working in nonlinear analysis and fixed point theory to have a rapid understanding.

$RM$ spaces and random normed spaces $($ briefly, $RN$ spaces $)$ were born in the course of the development of the theory of probabilistic metric spaces $($ briefly, $PM$ spaces $)$. The theory of $PM$ spaces was initiated by K. Menger in 1942 and subsequently founded by B. Schweizer, A. Sklar and the others, see \cite{SS8305} for a detailed historical survey on $PM$ spaces. A class of special $RM$ spaces $($ called uniform $RM$ spaces $)$ was first considered by A. \v{S}pac\v{e}k \cite{S55,S56}, the notion of a general $RM$ space was presented in \cite[Def.9.3.1]{SS8305}, where the random distance between two points in an $RM$ space is defined as a nonnegative random variable, similarly, the notion of an $RN$ space can be found in \cite[Chapt.15]{SS8305}. According to the tradition from the theory of $PM$ spaces, a $PM$ space is endowed with the $(\varepsilon,\lambda)$--topology introduced by B. Schweizer and A. Sklar in 1960. Therefore, an $RM$ space $($ regarded as a special $PM$ space$)$ and an $RN$ space $($ regarded as a spacial probabilistic normed space$)$ are often endowed with the $(\varepsilon,\lambda)$--topology. For a rather long time, the theory of $RN$ spaces did not obtain any substantial advances mainly because $RN$ spaces under the $(\varepsilon,\lambda)$--topology are not locally convex and even have trivial duals in general. The first important advance came in \cite{Guo89}, where Guo introduced the notion of an almost surely bounded random linear functional for $RN$ spaces and established the Hahn--Banach theorem for such random linear functionals. This leads to the theory of random conjugate spaces, whose further development also leads Guo to the notions of $RN$ modules and $RIP$ modules \cite{YZG91,Guo92}. The importance of $RN$ modules lies in that their module structure can make their random conjugate spaces and general continuous module homomorphisms on them so deeply developed that their theory is comparable to the corresponding theory of normed spaces, for example, Riesz's representation theorem on complete $RIP$ modules \cite{GY96} $($ where we should also mention Hansen and Richard's independent work on a class of spacial complete $RIP$ modules, called conditional Hilbert spaces constructed from the generalized conditional expectation, and its applications to finance \cite{HR87} $)$, the representation theorem of random conjugate spaces for a class of special $RN$ modules $L^0(\mathcal{F},B)$ \cite{Guo96}, the James' theorem characterizing random reflexivity of a complete $RN$ module \cite{GL05}, a random locally convex module as a random analogue of a locally convex space was subsequently presented and a separation theorem between a point and a closed $L^0$--convex subset was established in \cite{GXC09}, see \cite{Guo96a} for continuous module homomorphisms and applications of $RM$ spaces to probabilistic functional analysis initiated by \v{A}. \v{S}pac\v{e}k \cite{S55} and O. Han\v{s} \cite{H61}. Here, we should also mention R. Haydon, M. Levy and R. Raynaud's important work \cite{HLR91}, whose work is completely independent of the theory of $PM$ spaces and Guo's work, and who also presented the idea of $RN$ modules and established a lot of deep results. All the work on $RN$ modules before 2009 was developed under the $(\varepsilon,\lambda)$--topology.

The second important advance began with Filipovi\'{c}, Kupper and Vogeglpoth's work \cite{FKV09}. Motivated by financial applications, they introduced the notion of a locally $L^0$--convex module in 2009 in order to establish convex analysis over such a kind of topological module, see \cite{FKV09} for the rich financial background. Filipovi\'{c}, et.al's work \cite{FKV09} naturally leads to another kind of topology, called the locally $L^0$--convex topology, for a random locally convex module. Subsequently, Guo introduced the notion of $\sigma$--stability for a subset of an $L^0$--module and further established the inherent connection between some basic results derived from the two kinds of topologies--the $(\varepsilon,\lambda)$--topology and the locally $L^0$--convex topology for a random locally convex module, see \cite{Guo10} for details. Following Guo's work \cite{Guo10}, the subsequent development of random locally convex modules enters a new model, namely the theory of them was carried out under simultaneously considering the two kinds of topologies, see, e.g. \cite{GS11,GY12,Wu12,Wu13,GZZ14,Guo13,GZWYYZ17}, which in particular leads to a deep random convex analysis \cite{GZWYYZ17}.  Besides, the notion of $\sigma$--stability has played some crucial roles in a series of subsequent work, see, e.g.\cite{FM14a,FM14b,CKV15,DJKK16,DKKS13}. Finally, the notion of relative $\sigma$--stability was introduced independently by Wu and Guo \cite{WG15} and by Zapata \cite{Z17} and used to prove that the principle part of the theory of locally $L^0$--convex modules is equivalent to the theory of random locally convex modules endowed with the locally $L^0$--convex topology.

The remainder of this paper is organized as follows: Section \ref{section2} is devoted to discussing some basic problems closely related to $d$--$\sigma$--stability, for example, the connection between completeness with respect to the two kinds of uniformity induced by a random metric defined on a $d$--$\sigma$--stable $RM$ space. Section \ref{section3} can be regarded as applications of $d$--$\sigma$--stability, precisely speaking, Section \ref{section3} is first devoted to giving the precise form of Ekeland's variational principle on a $d$--$\sigma$--stable $RM$ space, and then to generalizing Nadler's fixed point theorem from a complete metric space to a complete $RM$ space, where a series of interesting corollaries of our result are given and the related known random fixed point theorems and concepts of random elements and random operators are mentioned when they are used.

\section{$d$--$\sigma$--stability for a subset of an $RM$ space}\label{section2}
In the sequel of this paper, $(\Omega, \mathcal{F}, P)$ always denotes a given probability space, $K$ the scalar field $R$ of real numbers or $C$ of complex numbers, $L^0(\mathcal{F}, K)$ the algebra of equivalence classes of $K$--valued $\mathcal{F}$--measurable random variables on $(\Omega, \mathcal{F}, P)$ and $\bar{L}^0(\mathcal{F})$ the set of equivalence classes of extended real--valued $\mathcal{F}$--measurable random variables on $(\Omega, \mathcal{F}, P)$. Specially, we simply write $L^0(\mathcal{F})$ for $L^0(\mathcal{F}, R)$.

Just as $\bar{R}:= [-\infty,+\infty]$ is a complete lattice under the ordinary total order $( R$ is Dedekind complete, namely the supremum or infimum principle $)$, it is well known from \cite{DS58} that $\bar{L}^0(\mathcal{F})$ is a complete lattice under the partial order $\leq : \xi \leq \eta$ if and only if $($ briefly, iff $)$ $\xi^0(\omega) \leq \eta^0(\omega)$ for almost all $\omega$ in $\Omega ~($ briefly, $\xi^0 \leq \eta^0$ a.s. $)$, where $\xi^0$ and $\eta^0$ are arbitrarily chosen representatives of $\xi$ and $\eta$, respectively. In particular, $L^0(\mathcal{F})$, as a sublattice of $\bar{L}^0(\mathcal{F})$, is Dedekind complete.

For any nonempty subset $H$ of $\bar{L}^0(\mathcal{F})$, $\bigvee H$ and $\bigwedge H$ stand for the supremum and infimum of $H$, respectively. Proposition \ref{proposition2.1} below surveys the nice properties of the lattice $\bar{L}^0(\mathcal{F})$, which will be frequently used in this paper.
\begin{proposition}\label{proposition2.1} \cite{DS58}.
Let $H$ be a nonempty subset of $\bar{L}^0(\mathcal{F})$, then the following hold:
\begin{enumerate}[(1)]
\item There exist two sequences $\{ a_n: n \in N \}$ and $\{ b_n: n \in N \}$ in $H$ such that $\bigvee_{n \geq 1} a_n = \bigvee H$ and $\bigwedge_{n \geq 1}b_n = \bigwedge H$.
\item If $H$ is directed upwards $($ or downwards $)$, then $\{ a_n: n \in N \}$ $($ correspondingly, $\{ b_n: n \in N \})$  in $($1$)$ can be chosen as nondecreasing $($ nonincreasing $)$.
\end{enumerate}
\end{proposition}

From now on, for any $A \in \mathcal{F}$, $I_A$ stands for the characteristic function of $A$, namely $I_A(\omega) = 1$ if $\omega \in A$, and 0 otherwise, $\tilde{I}_A$ denotes the equivalence class of $I_A$. Besides, we always make the following appointment:  $\xi > \eta$ on $A$ means $\xi^0 >\eta^0$ a.s. on $A$, where $A \in \mathcal{F}$ and $\xi^0$ and $\eta^0$ are arbitrarily chosen representatives of $\xi$ and $\eta$ in $\bar{L}^0(\mathcal{F})$, respectively.

Finally, we also employ the following notations:

$L^0_+(\mathcal{F}) = \{ \xi \in L^0(\mathcal{F})\ :\ \xi \geq 0 \}$;

$L^0_{++}(\mathcal{F}) = \{ \xi \in L^0(\mathcal{F})\ :\ \xi > 0$ on  $ \Omega  \}$.

To develop random metric spaces in the direction of functional analysis, Guo first adopted Definition \ref{definition2.2} below of an $RM$ space, which is an equivalent formulation of the original definition of an $RM$ space \cite[Def.9.3.1]{SS8305}. Similarly, we also adopt an equivalent formulation of the original definition of an $RN$ space \cite[p.240]{SS8305}. Please refer to \cite{Guo99} for the reason of changing the original formulation of $RM$ and $RN$ spaces.
\begin{definition}\label{definition2.2}
An ordered pair $(E,d)$ is called an $RM$ space with base $(\Omega, \mathcal{F}, P)$ if $E$ is a nonempty set and $d$ is a mapping from $E \times E$ to $L^0_+(\mathcal{F})$ such that the following axioms are satisfied:
\begin{enumerate}[(RM--1)]
\item $d(p,q) = 0$ if $p=q$;
\item $d(p,q) = d(q,p)$ for all $p$ and $q$ in $E$;
\item $d(p,q)=0$ implies $p=q$;
\item $d(p,r) \leq d(p,q) + d(q,r)$ for all $p,q,r \in E$.
\end{enumerate}
As usual, $d$ is called the random metric $($ or , distance $)$ on $E$, if $(RM-3)$ is not satisfied, then $d$ is called a random pseudometric on $E$.
\end{definition}
\begin{definition}\cite{Guo89,Guo92,Guo93,YZG91,Guo99}\label{definition2.3}
An ordered pair $(E,\|\cdot\|)$ is called an $RN$ space over $K$ with base $(\Omega, \mathcal{F},P)$ if $E$ is a linear space over $K$ and $\|\cdot\|$ is a mapping from $E$ to $L^0_+(\mathcal{F})$ such that the following axioms are satisfied:
\begin{enumerate}[(RN-1)]
\item $\|\alpha \cdot  x \| = |\alpha| \cdot \|x\|$, $\forall \alpha \in K $ and $x\in E$;
\item $\|x\| = 0$ implies $x = \theta$ $($the null of $E$$)$;
\item $\|x + y\| \leq \|x\| + \|y\|$, $ \forall x,y \in E$.
\end{enumerate}
As usual, $\|\cdot\|$ is called the random norm on $E$. If $(RN-2)$ is not satisfied, then $\|\cdot\|$ is called a random seminorm on $E$.

Furthermore, if $E$ is, in addition, a left module over the algebra $L^0(\mathcal{F},K)$ $($ briefly, an $L^0(\mathcal{F},K)$--module $)$ and the $RN$ space $(E,\|\cdot\|)$ also satisfies the following axiom:
\begin{enumerate}[(RNM-1)]
\item $\|\xi \cdot x \| = |\xi| \cdot \|x\|$, $\forall \xi \in L^0(\mathcal{F},K)$ and $x\in E$.
\end{enumerate}
Then the $RN$ space $(E,\|\cdot\|)$ is called an $RN$ module over $K$ with base $(\Omega, \mathcal{F},P)$.
\end{definition}
\begin{remark}\label{remark2.4}
In the theory of $RN$ modules, we always adopts the convention `` identifying any $\alpha \in K$ with $\alpha \cdot \tilde{I}_{\Omega}$", thus $K$ can be regarded as a subalgebra of $L^0(\mathcal{F},K)$. Since $\tilde{I}_{\Omega}$ is the unit element of $L^0(\mathcal{F},K)$, $\tilde{I}_{\Omega} \cdot x =x$, $\forall x \in E$ $($ according to the definition of a module over an algebra with the unit element $)$, $(RNM-1)$ naturally strengthens $(RN-1)$. In the latter literature \cite{FKV09}, a random norm satisfying $(RNM-1)$ is called an $L^0$--norm, correspondingly, a random seminorm satisfying $(RNM-1)$ is called an $L^0$--seminorm. Subsequently, we adopt these terminologies such as `` $L^0$--norm" and ``$L^0$--seminorm" for convenience and brevity.
\end{remark}

The following notion of $\sigma$--stability for a subset of an $L^0(\mathcal{F},K)$--module has played a crucial role in the development of $RN$ modules and random locally convex modules since 2010.
\begin{definition}\cite[Definition3.1]{Guo10}\label{definition2.5}
Let $E$ be an $L^0(\mathcal{F},K)$--module and $G \subset E$ a nonempty subset. $G$ is said to be  $\sigma$--stable $($ please notice: $G$ is said to have the countable concatenation property in the original terminology of \cite{Guo10} $)$ if, for each sequence $\{ x_n : n \in N \}$ in $G$ and each countable partition $\{ A_n : n \in N \}$ of $\Omega$ to $\mathcal{F}$, there exists some $x$ in $G$ such that $\tilde{I}_{A_n} \cdot x = \tilde{I}_{A_n} \cdot x_n$ for each $n \in N$.
\end{definition}

There is also a weaker notion than $\sigma$--stability, namely the notion of stability: a nonempty subset $G$ of an $L^0(\mathcal{F},K)$--module $E$ is said to be stable if $\tilde{I}_A \cdot x_1 + \tilde{I}_{A^c} \cdot x_2 \in G$ for any $A \in \mathcal{F}$ and any $x_1,x_2 \in G$. Let $(B,\|\cdot\|)$ be a Banach space over $K$ and $L^0(\mathcal{F},B)$ the $L^0(\mathcal{F},K)$--module of equivalence classes of $B$--valued strong random elements on $(\Omega,\mathcal{F},P)$, then $L^0(\mathcal{F},B)$ becomes an $RN$ module over $K$ with base $(\Omega,\mathcal{F},P)$ in a natural way $($ see \cite{Guo10} $)$, see Example \ref{example2.8} below for the notion of a strong random element. For a nonempty subset $G$ of $L^0(\mathcal{F},B)$, the notion of stability for $G$ was earlier considered in \cite{DS06} for $B = R^d$ and in \cite{M05} for $G \subset L^p(\mathcal{F},B)( 1\leq p <+\infty$, please notice $L^p(\mathcal{F},B) \subset L^0(\mathcal{F},B))$, namely $G$ is stable iff $\tilde{I}_A \cdot x + \tilde{I}_{A^c} \cdot y \in G$ for any $A \in \mathcal{F}$ and any $x,y \in G$, where $L^p(\mathcal{F},B)$ is the ordinary Lebesgue--Bochner function space. Just as pointed out by Guo in \cite{Guo10}, if $E$ in Def.2.5 is an $RN$ module $(E,\|\cdot\|)($ or more generally, a random locally convex module $)$, then $x$ in Def.2.5 must be unique, at this time we always write $\sum^{\infty}_{n=1}\tilde{I}_{A_n} \cdot x_n$ for $x$. It is obvious that $\tilde{I}_{A_n} \cdot x = \tilde{I}_{A_n} \cdot x_n$ for each $n \in N$ iff $\tilde{I}_{A_n} \cdot \|x-x_n\| =0$ for each $n \in N$. Motivated by this, A. Jamneshan, M. Kupper and J. M. Zapata recently introduced the notion of $d$--$\sigma$--stability for a nonempty subset of a random metric space in \cite{JKZ18}:
\begin{definition} \cite{JKZ18} \label{definition2.6}
Let $(E,d)$ be an $RM$ space with base $(\Omega,\mathcal{F},P)$ and $G$ a nonempty subset of $E$. $G$ is said to be $d$--$\sigma$--stable if, for each sequence $\{ x_n : n \in N \}$ in $G$ and each countable partition $\{ A_n : n \in N \}$ of $\Omega$ to $\mathcal{F}$, there exists some $x$ in $G$ such that $\tilde{I}_{A_n} \cdot d(x,x_n) =0$ for each $n \in N$.
\end{definition}

Similarly, $G$ is said to be $d$--stable if for any two elements $x_1$ and $x_2$ in $G$ and any $A \in \mathcal{F}$ there exists some $x \in G$ such that $\tilde{I}_{A} \cdot d(x,x_1) =0$ and $\tilde{I}_{A^c} \cdot d(x,x_2) =0$. Let $(E,d)$ be an $RM$ space with base $(\Omega,\mathcal{F},P)$ and $E$ also an $L^0(\mathcal{F},K)$--module such that $d(\tilde{I}_{A} \cdot x, \tilde{I}_{A} \cdot y) = \tilde{I}_{A} \cdot d(x,y)$ for any $A \in \mathcal{F}$ and $x,y \in E$, then it is easy to see that a nonempty subset $G$ of $E$ is $d$--$\sigma$--stable iff $G$ is $\sigma$--stable in the sense of \cite[Def.3.1]{Guo10}, in particular an $RN$ module $(E,\|\cdot\|)$ over $K$ with base $(\Omega, \mathcal{F},P)$ is such an $RM$ space under the random metric $d : E \times E \to L^0_+(\mathcal{F})$ defined by $d(x,y) = \|x-y\|, \forall x,y \in E$. Thus, we continue to employ the terminology `` $\sigma$--stability" for an $RN$ modules, which would not cause any confusion. Theorem \ref{theorem2.7} below shows that $x$ in Def.2.6 must be unique when $G$ is $d$--$\sigma$--stable $($ at which time $x$ is denoted by $\sum^{\infty}_{n=1}\tilde{I}_{A_n} \cdot x_n$ $)$ or when $G$ is $d$--stable $($ at which time $x$ is denoted by $\tilde{I}_{A} \cdot x_1 + \tilde{I}_{A^c} \cdot x_2~)$, and thus the requirement in \cite{JKZ18} that $x$ is unique is superfluous.

\begin{theorem}\label{theorem2.7}
Let $(E,d)$ and $G$ be the same as in Def. \ref{definition2.6}. Then we have the following assertions:
\begin{enumerate}[(1)]
\item When $G$ is $d$--$\sigma$--stable or $d$--stable, $x$ in Def. \ref{definition2.6} must be unique.
\item $G$ is $d$--stable iff for each positive integer $n$, each finite subset $\{ x_1, x_2, \cdots x_n \}$ of $G$ and each finite partition $\{ A_1, A_2, \cdots A_n \}$ of $\Omega$ to $\mathcal{F}$, there exists unique one $x$ in $G$ such that $\tilde{I}_{A_i} \cdot d(x,x_i) =0$ for each $i \in \{ 1, 2, \cdots n \}$.
\end{enumerate}
\end{theorem}
\begin{proof}
We only give the proof of $(1)$ for $d$--$\sigma$--stability, the proof of $(1)$ for $d$--stability is similar. Let $\{ x_n : n \in N \}$ and $\{ A_n : n \in N \}$ be the same as in Def.2.6, and further suppose that $x$ and $y$ are in $G$ such that $\tilde{I}_{A_n} \cdot d(x,x_n) =0$ and $\tilde{I}_{A_n} \cdot d(y,x_n) =0$ for each $n \in N$. Then by the triangle inequality $(RM-4)$ one has $\tilde{I}_{A_n} \cdot d(x,y) \leq \tilde{I}_{A_n} \cdot d(x,x_n) + \tilde{I}_{A_n} \cdot d(x_n,y) =0$ for each $n \in N$, which further implies that $d(x,y) = (\sum^{\infty}_{n=1}\tilde{I}_{A_n}) \cdot d(x,y) = \sum^{\infty}_{n=1}\tilde{I}_{A_n} \cdot d(x,y) =0$, namely $x=y$ by $(RM-3)$.

(2). It only needs to prove necessity. The necessity for $n=1$ holds trivially, and $d$--stability of $G$ amounts to the necessity for $n=2$. We will complete our proof by induction method, for this we suppose that $n \geq 2$ is a positive integer such that the necessity holds for any positive integer $l \leq n$. Then for any finite subset $\{ x_1, x_2, \cdots x_{n+1}\}$ of $G$ and any finite partition $\{ A_1, A_2, \cdots A_{n+1} \}$ of $\Omega$ to $\mathcal{F}$, there exists unique one $y$ in $G$ such that $y = \sum^{n-1}_{i=1}\tilde{I}_{A_i} \cdot x_i + \tilde{I}_{A_n \cup A_{n+1}} \cdot x_n$ and there exists unique one $x$ in $G$ such that $x = \tilde{I}_{A_1 \cup A_2 \cdots \cup A_n} \cdot y +\tilde{I}_{A_{n+1}} \cdot x_{n+1}$, namely, one has the following:
\begin{enumerate}[(I)]
\item $\tilde{I}_{A_i} \cdot d(y,x_i)=0$ for each $i \in \{ 1, 2, \cdots n-1 \}$;
\item $\tilde{I}_{A_n \cup A_{n+1}} \cdot d(y,x_n)=0$;
\item $\tilde{I}_{A_1 \cup A_2 \cdots \cup A_n} \cdot d(x,y)=0$ and $\tilde{I}_{A_{n+1}} \cdot d(x,x_{n+1})=0$.
\end{enumerate}

By (I) and the first equality of (\uppercase\expandafter{\romannumeral3}), one has $\tilde{I}_{A_i} \cdot d(y,x_i)=0$ and $\tilde{I}_{A_i} \cdot d(x,y)=0$ for each $i \in \{ 1, 2, \cdots n-1 \}$, so that $\tilde{I}_{A_i} \cdot d(x,x_i)=0$ for each $i \in \{ 1, 2, \cdots n-1 \}$. By (II) and the first equality of (III), one has $\tilde{I}_{A_n} \cdot d(y,x_n)=0$ and $\tilde{I}_{A_n} \cdot d(x,y)=0$, namely $\tilde{I}_{A_n} \cdot d(x,x_n)=0$. This, combined with the second equality of (III), shows that there exists unique one $x$ in $G$ such that $\tilde{I}_{A_i} \cdot d(x,x_i)=0$ for each $i \in \{ 1, 2, \cdots n+1 \}$, which ends the proof of necessity.
\end{proof}

Following are two nontrivial examples of $d$--$\sigma$--stable sets.
\begin{example}\label{example2.8}
Let $(M,d)$ be a metric space. A mapping $V$ from $(\Omega,\mathcal{F},P)$ to $M$ is called a random element if $V^{-1}(G) = \{ \omega \in \Omega : V(\omega) \in G \} \in \mathcal{F}$ for any $d$--open set $G$ of $M$. A random element $V$ is said to be simple if $V$ only takes finitely many values. Further, a random element $V$ is said to be strong if $V$ is the pointwise limit of a sequence of simple random elements. It is known from \cite{BR76} that a random element is strong iff its range is a separable subset of $M$, and thus when $(M,d)$ is separable the notion of a random element coincides with that of a strong random element. Let $L^0(\mathcal{F},M)$ be the set of equivalence classes of strong random element from $(\Omega,\mathcal{F},P)$ to $(M,d)$. For any $x$ and $y$ in $L^0(\mathcal{F},M)$, let $x^0$ and $y^0$ be respectively arbitrarily chosen representatives of $x$ and $y$, and let $d(x,y)$ denote the equivalence class of $d(x^0(\cdot),y^0(\cdot))$, then $(L^0(\mathcal{F},M),d)$ becomes an $RM$ space with base $(\Omega,\mathcal{F},P)$ and $L^0(\mathcal{F},M)$ is always $d$--$\sigma$--stable. In fact, for any sequence $\{ x_n : n \in N \}$ in $L^0(\mathcal{F},M)$, arbitrarily choose a representative $x^0_n$ of $x_n$ for each $n \in N$, then for any countable partition $\{ A_n : n \in N \}$ of $\Omega$ to $\mathcal{F}$, define $x^0 : \Omega \to M$ by $x^0(\omega) = x^0_n(\omega)$ when $\omega \in A_n$, it is easy to see that $x = \sum^{\infty}_{n=1}\tilde{I}_{A_n} \cdot x_n$, $($ where $x$ is the equivalence class of $x^0 )$, namely $\tilde{I}_{A_n} \cdot d(x,x_n)=0$ for each $n \in N$.
\end{example}
\begin{example} \label{example2.9}
Let $(M,d)$ be a complete separable metric space and $V : \Omega \to 2^M$ satisfy the following conditions: $V(\omega)$ is closed and nonempty for any $\omega \in \Omega$ and $V^{-1}(G) = \{ \omega \in \Omega : V(\omega) \cap G \neq \emptyset \} \in \mathcal{F}$ for any $d$--open set $G$ of $M$. It follows from \cite{H75} that $V$ has a measurable selection $\xi$, namely $\xi$ is a random element and $\xi(\omega) \in V(\omega)$ for each $\omega \in \Omega$. Let $H = \{ x \in L^0(\mathcal{F},M) : x$ is the equivalence class of some measurable selection of $V \}$, then similarly to Example \ref{example2.8} it can be proved that $H$ is a $d$--$\sigma$--stable subset of $L^0(\mathcal{F},M)$.
\end{example}
\begin{theorem}\label{theorem2.10}
Let $(E,d)$ be an $RM$ space with base $(\Omega,\mathcal{F},P)$ and $G$ a nonempty subset of $E$. Then we have the following statements:
\begin{enumerate}[(1)]
\item If $G$ is $d$--stable, then $L := \{ d(x,y) : (x,y) \in G \times G \}$ satisfies: $d(x_1,y_1) \bigwedge d(x_2,y_2) \\ \in L$ and $d(x_1,y_1) \bigvee d(x_2,y_2) \in L$ for any $(x_1,y_1)$ and $(x_2,y_2) \in G \times G$.
\item If $G$ is $d$--stable, then for any fixed $x_0 \in E$, $L := \{ d(x_0,y) : y \in G \}$ satisfies: $d(x_0,y_1)\bigwedge d(x_0,y_2) \in L$ and $d(x_0,y_1)\bigvee d(x_0,y_2) \in L$ for any $y_1$ and $y_2 \in G$.
\item If $G$ is $d$--$\sigma$--stable, then $L := \{ d(x,y) : (x,y) \in G \times G \}$ is a $\sigma$--stable subset of $L^0(\mathcal{F})$.
\item If $G$ is $d$--$\sigma$--stable, then for any fixed $x_0 \in E$, $L = \{ d(x_0,y) : y \in G \}$ is a $\sigma$--stable subset of $L^0(\mathcal{F})$.
\end{enumerate}
\end{theorem}
\begin{proof}
(1). let $(x_1,y_1)$ and $(x_2,y_2) \in G \times G$ and $A =\{ \omega \in \Omega : d^0(x_1,y_1)(\omega) \leq d^0(x_2,y_2)(\omega) \}$, where $d^0(x_1,y_1)$ and $d^0(x_2,y_2)$ are respectively arbitrarily chosen representatives of $d(x_1,y_1)$ and $d(x_2,y_2)$. Then $d(x_1,y_1) \bigwedge d(x_2,y_2) = \tilde{I}_{A} \cdot d(x_1,y_1) + \tilde{I}_{A^c} \cdot d(x_2,y_2)$. Since $G$ is $d$--stable, there exist unique $x$ and $y$ in $G$ such that $x = \tilde{I}_{A} \cdot x_1 + \tilde{I}_{A^c} \cdot x_2$ and $y = \tilde{I}_{A} \cdot y_1 + \tilde{I}_{A^c} \cdot y_2$, namely, one has the following:
\begin{enumerate}[(i)]
\item $\tilde{I}_{A} \cdot d(x,x_1) = 0, \tilde{I}_{A^c} \cdot d(x,x_2) = 0$;
\item $\tilde{I}_{A} \cdot d(y,y_1) = 0, \tilde{I}_{A^c} \cdot d(y,y_2) = 0$.
\end{enumerate}

By the triangle inequality one can obtain $\tilde{I}_{A} \cdot d(x,y) = \tilde{I}_{A} \cdot d(x_1,y_1)$ and $\tilde{I}_{A^c} \cdot d(x,y) = \tilde{I}_{A^c} \cdot d(x_2,y_2)$, so $d(x,y) = \tilde{I}_{A} \cdot d(x,y) + \tilde{I}_{A^c} \cdot d(x,y) = \tilde{I}_{A} \cdot d(x_1,y_1) + \tilde{I}_{A^c} \cdot d(x_2,y_2) = d(x_1,y_1) \bigwedge d(x_2,y_2)$. Similarly, one also has $d(x_1,y_1) \bigvee d(x_2,y_2) \in L$.

(2). Proof is similar to that of (1) and more simple, so is omitted.

(3). Let $\{ (x_n,y_n) : n \in N \}$ be any sequence in $G \times G$ and $\{ A_n : n \in N \}$ any countable partition of $\Omega$ to $\mathcal{F}$. Since $G$ is $d$--$\sigma$--stable, there exist unique $x$ and $y$ in $G$ such that $x = \sum^{\infty}_{n=1}\tilde{I}_{A_n} \cdot x_n$ and $y = \sum^{\infty}_{n=1}\tilde{I}_{A_n} \cdot y_n$, namely $\tilde{I}_{A_n} \cdot d(x,x_n)=0$ and $\tilde{I}_{A_n} \cdot d(y,y_n)=0$ for each $n \in N$. Again by the triangle inequality one can have $\tilde{I}_{A_n} \cdot d(x,y) = \tilde{I}_{A_n} \cdot d(x_n,y_n)$ for each $n \in N$, so $d(x,y) = (\sum^{\infty}_{n=1}\tilde{I}_{A_n}) \cdot d(x,y) = \sum^{\infty}_{n=1}\tilde{I}_{A_n} \cdot d(x,y) = \sum^{\infty}_{n=1}\tilde{I}_{A_n} \cdot d(x_n,y_n)$.

(4). Proof is similar to that of (3) and more easy, so is omitted.
\end{proof}

The following idea of introducing $(\varepsilon,\lambda)$--uniformity for an $RM$ space is due to B. Schweizer and A. Sklar \cite{SS8305} and that of introducing $L^0$--uniformity is due to D. Filipovi\'{c}, et.al \cite{FKV09}.

\begin{definition}\label{definition2.11} \cite{GY12}
Let $(E,d)$ be an $RM$ space with base $(\Omega,\mathcal{F},P)$. For any positive numbers $\varepsilon > 0$ and $0 < \lambda <1$, let $U(\varepsilon,\lambda) = \{ (x,y) \in E \times E : P\{\omega \in \Omega : d(x,y)(\omega) < \varepsilon \} > 1-\lambda \}$, then $\mathcal{U} = \{ U(\varepsilon,\lambda) : \varepsilon >0, ~0 <\lambda <1 \}$ forms a base for some Hausdorff uniformity on $E$, called the $(\varepsilon,\lambda)$--uniformity induced by $d$, the corresponding topology is called the $(\varepsilon,\lambda)$--topology, denoted by $\mathcal{T}_{\varepsilon,\lambda}$. For any $\varepsilon \in L^0_{++}(\mathcal{F})$, let $U(\varepsilon) = \{ (x,y) \in E \times E : d(x,y) \leq \varepsilon \}$, then $\mathcal{U}_{L^0} = \{ U(\varepsilon) : \varepsilon \in L^0_{++}(\mathcal{F}) \}$ forms a base for some Hausdorff uniformity on $E$, called the $L^0$--uniformity induced by $d$, the corresponding topology is called the $L^0$--topology, denoted by $\mathcal{T}_c$. $(E,d)$ is said to be $(\varepsilon,\lambda)$--complete and $L^0$--complete if it is complete with respect to the $(\varepsilon,\lambda)$--uniformity and $L^0$--uniformity, respectively.

Similarly, an $RM$ space $(E,d)$ always has an $(\varepsilon,\lambda)$--completion and an $L^0$--completion with respect to the two kinds of uniformities stated above, denoted by $(\tilde{E}_{\varepsilon,\lambda},d)$ and $(\tilde{E}_{c},d)$, respectively, which are both unique in the sense of isometric isomorphism with respect to random metric.
\end{definition}

We would also like to point out that the $(\varepsilon,\lambda)$--uniformity is always metrizable but the $L^0$--uniformity is not necessarily metrizable. On the other hand, the $L^0$--uniformity is generally much stronger than $(\varepsilon,\lambda)$--uniformity, so an $RM$ space must be $L^0$--complete if it is $(\varepsilon,\lambda)$--complete, we will prove that the two kinds of completeness coincide when the $RM$ space is $d$--$\sigma$--stable, see Theorem \ref{theorem2.13} below. To prove Theorem \ref{theorem2.13}, let us first establish a result similar to \cite[Theorem3.12]{Guo10}.

\begin{theorem}\label{theorem2.12}
Let $(E,d)$ be an $RM$ space with base $(\Omega,\mathcal{F},P)$ and $G$ a nonempty subset of $E$. Then we have the following statements:
\begin{enumerate}[(1)]
\item $d(x,G) = d(x,G^{-}_{\varepsilon,\lambda}) = d(x,G^{-}_{c})$ for any $x \in E$, where $G^{-}_{\varepsilon,\lambda}$ and $G^{-}_{c}$ stand for the closures of $G$ with respect to $\mathcal{T}_{\varepsilon,\lambda}$ and $\mathcal{T}_{c}$, respectively, and $d(x,H) = \bigwedge \{ d(x,h) : h \in H \}$ for any nonempty subset $H$ of $E$.
\item If $G$ is $d$--stable, then $d(x,G)= 0$ iff $x \in G^{-}_{\varepsilon,\lambda}$.
\item If $G$ is $d$--$\sigma$--stable, then $d(x,G)= 0$ iff $x \in G^{-}_{c}$, in particular at this time $G^{-}_{\varepsilon,\lambda} = G^{-}_{c}$.
\end{enumerate}
\end{theorem}
\begin{proof}
\begin{enumerate}[(1)]
\item It is obvious, so is omitted.
\item Sufficiency is obvious by (1). For the proof of necessity, suppose $d(x,G)=0$. Since $\{ d(x,g) : g \in G \}$ is directed downwards by (1) of Theorem \ref{theorem2.10}, there exists a sequence $\{ g_n : n \in N \}$ in $G$ by (2) of Proposition \ref{proposition2.1} such that $\{ d(x,g_n) : n \in N \}$ converges a.s. to 0 in a nonincreasing way, and hence also converges in probability to 0, namely $\{ g_n : n \in N \}$ converges in $\mathcal{T}_{\varepsilon,\lambda}$ to $x$, that is to say, $x \in G^{-}_{\varepsilon,\lambda}$.
\item Sufficiency is obvious by (1). For the proof of necessity, let us first notice that $d$--$\sigma$--stable of $G$ also implies its $d$--stability, then there exists a sequence $\{ g_n : n \in N \}$ in $G$ as in the proof of (2) such that $\{ d(x,g_n) : n \in N \}$ converges a.s. to 0 in a nonincreasing way. Now, we prove $x \in G^{-}_{c}$ as follows. For a given $\varepsilon \in L^0_{++}(\mathcal{F})$, let $\varepsilon^0$ be an arbitrarily chosen representative of $\varepsilon$ and $d^0(x,g_n)$ that of $d(x,g_n)$ for each $n$, we can, without loss of generality, assume that $\bigcup_{n \geq 1}B_n = \Omega$ and $B_n \subset B_{n+1}$ for each $n \in N$, where $B_n = \{ \omega \in \Omega : d^0(x,g_n)(\omega) \leq \varepsilon^0(\omega) \}$ for each $n \in N$. Let $A_n = B_n \setminus B_{n-1}$ for each $n \geq 1$ (with $B_0 = \emptyset$ ) and $g = \sum_{n \geq 1} \tilde{I}_{A_n} \cdot g_n$, then it is easy to check that $d(x,g) = \sum_{n \geq 1} \tilde{I}_{A_n} \cdot d(x,g_n) \leq \varepsilon$, which means that $x \in G^{-}_c$.
\end{enumerate}
\end{proof}
\begin{theorem}\label{theorem2.13}
Let $(E,d)$ be a $d$--$\sigma$--stable $RM$ space. Then $E$ is $(\varepsilon,\lambda)$--complete iff $E$ is $L^0$--complete.
\end{theorem}
\begin{proof}
Necessity is obvious since the $(\varepsilon,\lambda)$--uniformity is weaker than the $L^0$--uniformity. For sufficiency, let $\tilde{E}_{\varepsilon,\lambda}$ be the $(\varepsilon,\lambda)$--completion, then $\tilde{E}_{\varepsilon,\lambda}= E^{-}_{\varepsilon,\lambda} = E^{-}_c = E$ by (3) of Theorem \ref{theorem2.12}.
\end{proof}
\begin{remark}\label{remark2.14}
For an $RN$ module or ( more generally ) a random locally convex module, its $(\varepsilon,\lambda)$--completeness already implies its $\sigma$--stability, so Theorem 3.18 of \cite{Guo10} can also be stated in the way: a random locally convex module is $\mathcal{T}_{\varepsilon,\lambda}$--complete iff it is $\mathcal{T}_{c}$--complete and has $\sigma$--stability. But an $(\varepsilon,\lambda)$--complete $RM$ space is not necessarily $d$--$\sigma$--stable, so Theorem \ref{theorem2.13} does not possess such a tidy statement! But we have the following interesting Theorem \ref{theorem2.15}.
\end{remark}
\begin{theorem}\label{theorem2.15}
Let $(E,d)$ be an $(\varepsilon,\lambda)$--complete $RM$ space with base $(\Omega,\mathcal{F},P)$. Then $E$ is $d$--stable iff $E$ is $d$--$\sigma$--stable.
\end{theorem}
\begin{proof}
We only need to prove necessity since sufficiency is obvious. For this, let $\{ x_n : n \in N \}$ be a sequence in $E$ and $\{ A_n : n \in N \}$ a countable partition of $\Omega$ to $\mathcal{F}$, and further fix an element $y_0 \in E$, then there exists unique one $g_n \in E$ for each $n \in N$ such that $g_n = \sum^n_{i=1} \tilde{I}_{A_i} \cdot x_i + \tilde{I}_{(\cup^n_{i=1}A_i)^c} \cdot y_0$, namely we have the following:
\begin{enumerate}[(1)]
\item $\tilde{I}_{A_i} \cdot d(g_n,x_i)=0$ for each $i \in \{ 1,2,\cdots n \}$.
\item $\tilde{I}_{(\cup^n_{i=1}A_i)^c} \cdot d(g_n,y_0)=0$.
\end{enumerate}

By the triangle inequality and (1) one has $\tilde{I}_{A_i} \cdot d(g_n,g_{n+k})=0$ for each $i \in \{ 1,2,\cdots n \}$ and each $k \in N$, then $P\{ \omega \in \Omega : d(g_n,g_{n+k})(\omega) > \varepsilon \} \leq P (\cup^\infty_{i=n+1} A_i) = \sum^\infty_{i=n+1} P(A_i)$ for each positive number $\varepsilon$, which means $\{ g_n : n \in N \}$ is an $(\varepsilon,\lambda)$--Cauchy sequence in $E$. By the $(\varepsilon,\lambda)$--completeness of $E$, there exists unique one $x \in E$ such that $\{ d(g_n,x) : n \in N \}$ converges in probability to 0. Since, for each given $n \in N$, $\{ d(g_n,g_{n+k}) : k \in N \}$ converges in probability to $d(g_n,x)$, $\tilde{I}_{A_i} \cdot d(g_n,x)=0$ for each $i \in \{ 1,2,\cdots n \}$. Again by (1) and the triangle inequality one has $\tilde{I}_{A_i} \cdot d(x,x_i) =0$ for each $n \in N$ and each $i \in \{ 1,2,\cdots n \}$, namely $\tilde{I}_{A_i} \cdot d(x,x_i) =0$ for each $i \in N$.
\end{proof}

The following two examples shows that the $(\varepsilon,\lambda)$--completeness and $d$--$\sigma$--stability for an $RM$ space do not imply each other.
\begin{example}\label{example2.16}
Let $\Omega =[0,1]$, $\mathcal{F}=$ the $\sigma$--algebra of Lebesgue measurable subsets of $[0,1]$ and $P=$ the Lebesgue measure on $\mathcal{F}$. Since $L^0(\mathcal{F})$ is a $\mathcal{T}_{\varepsilon,\lambda}$--complete $RN$ module, let $E=$ the set of equivalence classes of Lebesgue measurable functions ( on $[0,1]$ ) taking countably many values in $R$, then $E$ is obviously a $\sigma$--stable subset of $L^0(\mathcal{F})$, namely $(E,d)$ is $d$--$\sigma$--stable under the random metric $d$ defined by $d(\xi,\eta) = |\xi -\eta|$ for any $\xi$ and $\eta \in E$, but $(E,d)$ is not $(\varepsilon,\lambda)$--complete since $E$ is a denes subset of $L^0(\mathcal{F})$.
\end{example}
\begin{example}
Let $(\Omega,\mathcal{F},P)$ be the same as in Example \ref{example2.16}, $I$ the equivalence class of the identity function on $[0,1]$ and $E = \{ \alpha \cdot I : \alpha \in R \}$. Then $E$ is an $(\varepsilon,\lambda)$--complete $RM$ space as a subspace of $L^0(\mathcal{F})$, but it is not difficult to check $(E,d)$ is not $d$--$\sigma$--stable, where $d(\xi,\eta) = |\xi -\eta|$ for any $\xi,\eta \in E$.
\end{example}

\section{Applications of $d$--$\sigma$--stability}\label{section3}
\par
Just as in the classical case of metric spaces, let $(E_1,d_1)$ and $(E_2,d_2)$ be two $RM$ spaces with base $(\Omega,\mathcal{F},P)$, then $(E_1\times E_2,d)$ is still an $RM$ space with base $(\Omega,\mathcal{F},P)$, where $d((x_1,y_1),(x_2,y_2))=d_1(x_1,x_2)+d_2(y_1,y_2)$ for any $(x_1,y_1),(x_2,y_2)\in E_1\times E_2$. It is also easy to see that the $(\varepsilon,\lambda)$--topology and $L^0$--topology on $ E_1\times E_2$ induced by $d$ are just the product topologies $\mathcal{T}^1_{\varepsilon,\lambda}\times \mathcal{T}^2_{\varepsilon,\lambda}$ and $\mathcal{T}^1_c\times \mathcal{T}^2_c$, respectively, where $\mathcal{T}^i_{\varepsilon,\lambda}$ and $\mathcal{T}^i_c$ are the $(\varepsilon,\lambda)$--topology and $L^0$--topology on $E^i$, respectively, $i=1,2$. In this section, we only involves the product of an $RM$ space $(E,d)$ with base $(\Omega,\mathcal{F},P)$ and $L^0(\mathcal{F})$, $L^0(\mathcal{F})$ is an $RN$ module and, of course, an $RM$ space with base $(\Omega,\mathcal{F},P)$ under the random metric $d_2$ defined by $d_2(\xi,\eta)=|\xi-\eta|$ for any $\xi$ and $\eta$ in $L^0(\mathcal{F})$, it is obviously $\sigma$--stable.
\par
Lemma \ref{lemma3.1} below was first obtained in \cite{GZWG18} as a direct corollary of Theorem 3.12 of \cite{Guo10}, which will play a key role in several spaces of this section.
\begin{lemma}\label{lemma3.1}\cite{GZWG18}
Let $G$ be a $\sigma$--stable subset of $L^0(\mathcal{F})$ such that $G$ has an upper bound $($or a lower bound$)$ in $L^0(\mathcal{F})$. Then for any $\varepsilon\in L^0_{++}(\mathcal{F})$ there exists some $g\in G$ such that $g>\bigvee G-\varepsilon$ on $\Omega$ $($accordingly, $g<\bigwedge G+\varepsilon$ on $\Omega)$.
\end{lemma}

\begin{definition}\label{definition3.2}\cite{GY12}
Let $(E,d)$ be an $RM$ space with base $(\Omega,\mathcal{F},P)$ and $f$ a mapping from $E$ to $\bar{L}^0(\mathcal{F})$. $f$ is said to be proper if $f(x)>-\infty$ on $\Omega$ for any $x\in E$ and $dom(f):=\{x\in E: f(x)<+\infty$ on $\Omega\}\neq\emptyset$; further, a proper $f: E\rightarrow \bar{L}^0(\mathcal{F})$ is said to be $\mathcal{T}_{\varepsilon,\lambda}$--lower semicontinuous $($or $\mathcal{T}_c$--lower semicontinuous$)$ if $epi(f):=\{(x,r)\in E\times L^0(\mathcal{F}) : f(x)\leq r\}$ is $\mathcal{T}_{\varepsilon,\lambda}$--closed in $E\times L^0(\mathcal{F})$ $($respectively, $\mathcal{T}_c$--closed in $E\times L^0(\mathcal{F}))$; finally, if $E$ is $d$--$\sigma$--stable, $f$ is said to be $\sigma$--stable if $f(\sum_{n=1}^{\infty}\tilde{I}_{A_n}\cdot x_n)=\sum_{n=1}^{\infty}\tilde{I}_{A_n}\cdot f(x_n)$ for each sequence $\{x_n: n\in N\}$ in $E$ and each countable partition of $\Omega$ to $\mathcal{F}$. Similarly, if $E$ is $d$--stable, $f$ is said to be stable if, for any $A\in \mathcal{F}$ and any $x_1,x_2\in E$, $f(\tilde{I}_A\cdot x_1+\tilde{I}_{A^c}\cdot x_2)=\tilde{I}_A\cdot f(x_1)+ \tilde{I}_{A^c}\cdot f(x_2)$.
\end{definition}

\begin{remark}\label{remark3.3}
Similarly to the proof of $(2)$ of Theorem \ref{theorem2.7}, one can prove that $f: E \rightarrow \bar{L}^0(\mathcal{F})$ is stable iff $f(\sum^n_{i=1}\tilde{I}_{A_i}\cdot x_i)=\sum^n_{i=1}\tilde{I}_{A_i}\cdot f(x_i)$ for each $n\in N$, each finite subset $\{x_1,x_2,\cdots x_n\}$ in $E$ and each finite partition $\{A_1,A_2,\cdots A_n\}$ of $\Omega$ to $\mathcal{F}$.
\end{remark}

Local functions defined on an $L^0(\mathcal{F})$--module often occurs in the study of financial problems, see, e.g.\cite{FKV09,FM14a,FM14b,GZWYYZ17,GZZ14}. Let us recall: a function $f$ from an $L^0(\mathcal{F})$--module $E$ to $\bar{L}^0(\mathcal{F})$ is said to be local if $\tilde{I}_A\cdot f(\tilde{I}_A\cdot x)=\tilde{I}_A\cdot f(x)$ for any $A\in \mathcal{F}$ and any $x\in E$. Obviously, when $(E,\|\cdot\|)$ is an $RN$ module, $f$ is stable iff $f$ is local $($ notice: an $RN$ module $E$ is always stable $)$, further when $E$ is a $\sigma$--stable $RN$ module, it is easy to see that $f$ is $\sigma$--stable iff $f$ is local. Thus Theorem \ref{theorem3.4} below is a good generalization of Theorem 3.5 of \cite{GY12}.

\begin{theorem}\label{theorem3.4}
Let $(E,d)$ be an $RM$ space with base $(\Omega,\mathcal{F},P)$, $G$ a $d$--$\sigma$--stable subset of $E$ and $f: G\rightarrow \bar{L}^0(\mathcal{F})$ proper, $\sigma$--stable and bounded from below, namely there exists some $\eta\in L^0(\mathcal{F})$ such that $f(x)\geq\eta$ for any $x\in G$, then for each $\varepsilon\in L^0_{++}(\mathcal{F})$ there exists some $x_{\varepsilon}\in G$ such that $f(x_{\varepsilon}) \leq \bigwedge f(G)+\varepsilon$. Similarly, if $f: G\rightarrow \bar{L}^0(\mathcal{F})$ is proper $\sigma$--stable and bounded from above, then for each $\varepsilon\in L^0_{++}(\mathcal{F})$ there exists some $x_{\varepsilon}\in G$ such that $f(x_{\varepsilon}) \geq \bigvee f(G)-\varepsilon$.
\end{theorem}

\begin{proof}
Let $f$ be bounded from below and $dom(f) = \{ x \in G : f(x) < +\infty$ on $\Omega \}$, then $dom(f) = \{ x \in G : f(x) \in L^0(\mathcal{F})$, and is nonempty since $f$ is proper. Further, it is obvious that $f(dom(f))$ is bounded from below and $\bigwedge f(dom(f)) = \bigwedge f(G)$, it remains to check that $f(dom(f))$ is $\sigma$--stable, in fact, let $\{ x_n : n \in N \}$ be a sequence in $dom(f)$ and $\{ A_n : n \in N \}$ a countable partition of $\Omega$ to $\mathcal{F}$, then there exists unique one $x \in G$ such that $x = \sum^\infty_{n=1} \tilde{I}_{A_n} \cdot x_n$, it follows from $f(x) = \sum^\infty_{n=1} \tilde{I}_{A_n} \cdot f(x_n)$ that $f(x) \in L^0(\mathcal{F})$, namely $x \in dom(f)$, which also implies that $f(x) \in f(dom(f))$. Thus there exists $x_{\varepsilon} \in dom(f) \subset G$ by Lemma \ref{lemma3.1} such that $f(x_{\varepsilon}) \leq \bigwedge f(dom(f)) + \varepsilon = \bigwedge f(G) + \varepsilon$.

Finally, if $f$ is bounded from above, then $f(G) \subset L^0(\mathcal{F})$ since $f$ is proper. Similarly to the proof of $\sigma$--stability of $f(dom(f))$ as above, one can easily see that $f(G)$ is also $\sigma$--stable in $L^0(\mathcal{F})$, again by Lemma \ref{lemma3.1} there exists $x_{\varepsilon} \in G$ such that $f(x_{\varepsilon}) \geq \bigvee f(G) - \varepsilon$.
\end{proof}

\par
Theorem \ref{theorem3.5} below is a good generalization of Proposition 3.8 of \cite{GY12}.

\begin{theorem}\label{theorem3.5}
Let $(E,d)$ be a $d$--$\sigma$--stable $RM$ space with base $(\Omega,\mathcal{F},P)$ and $f: E\rightarrow \bar{L}^0(\mathcal{F})$ a proper and $\sigma$--stable function. Then $f$ is $\mathcal{T}_{\varepsilon,\lambda}$--lower semicontinuous iff $f$ is $\mathcal{T}_c$--lower semicontinuous.
\end{theorem}

\begin{proof}
Since $epi(f)=\{(x,r)\in E\times L^0(\mathcal{F}) : f(x)\leq r\}$ is clearly $d$--$\sigma$--stable in $E\times L^0(\mathcal{F})$, it follows from $(3)$ of Theorem \ref{theorem2.12} that $epi(f)$ is $\mathcal{T}_{\varepsilon,\lambda}$--closed iff it is $\mathcal{T}_c$--closed.
\end{proof}

\par
With the notion of $d$--$\sigma$--stability, we are able to generalize Theorems 3.6, 3.10 and 3.11 of \cite{GY12} to Theorems \ref{theorem3.6}, \ref{theorem3.7} and \ref{theorem3.8} below which are very concise and even comparable to the corresponding classical results in metric spaces\cite{Eke74}.

\begin{theorem}\label{theorem3.6}
Let $(E,d)$ be a $d$--$\sigma$--stable $(\varepsilon,\lambda)$--complete $RM$ space with base $(\Omega,\mathcal{F},P)$, $\varepsilon\in L^0_{++}(\mathcal{F})$ and $f: E\rightarrow \bar{L}^0(\mathcal{F})$ proper, $\sigma$--stable, $\mathcal{T}_{\varepsilon,\lambda}$--lower semicontinuous and bounded from below. Then for any given point $x_0\in E$ satisfying $f(x_0)\leq\bigwedge f(E)+\varepsilon$ and any given $\alpha\in L^0_{++}(\mathcal{F})$, there exists $z\in E$ such that the following hold:
\begin{enumerate}[$(1)$]
\item $f(z)\leq f(x_0)-\alpha\cdot d(z,x_0)$;
\item $\|z-x_0\|\leq \alpha^{-1}\cdot \varepsilon$;
\item For each $x\in E$ such that $x\neq z, f(x)>f(z)-\alpha\cdot d(x,z)$, where $``>"$ means $``\geq"$ and $``\neq"$.
\end{enumerate}
\end{theorem}

\begin{proof}
It follows from Theorem 2.12 of \cite{GY12} that $(1), (2)$ and $(3)'$ below hold: \\
$(3)'$. For each $x\in E$ such that $x\neq z, f(x)+ \alpha d(x,z)\nleqslant f(z)$.
\par
If there exists some $v\in E$ with $v\neq z$ such that $(3)$ is not true. If $f(v)=f(z)-\alpha\cdot d(v,z)$, this contradicts $(3)'$. If $\alpha\cdot d(z,v)+f(v)\neq f(z)$, then $P(A)>0$, where $A=\{\omega\in \Omega: \xi^0(\omega)<\eta^0(\omega)\}$ and $\xi^0$ and $\eta^0$ are arbitrarily chosen representatives of $\alpha\cdot d(z,v)+f(v)$ and $f(z)$. Let $\bar{v}=\tilde{I}_A\cdot v+\tilde{I}_{A^c}\cdot z$, then we have the following two assertions:\\
$(4)$. $\bar{v}\neq z$;\\
$(5)$. $\alpha\cdot d(z,\bar{v})+f(\bar{v})\leq f(z)$.
\par
In fact, if $\bar{v}=z,$ then $\tilde{I}_A\cdot d(z,v)=0$, further $f(\bar{v})=\tilde{I}_A\cdot f(v)+\tilde{I}_{A^c}\cdot f(z)$ implies $\tilde{I}_A\cdot f(z)=\tilde{I}_A\cdot f(\bar{v})=\tilde{I}_A\cdot f(v)$ and $\tilde{I}_{A^c}\cdot f(z)=\tilde{I}_{A^c}\cdot f(\bar{v})$, so $\tilde{I}_A\cdot \alpha\cdot d(v,z)+ \tilde{I}_A\cdot f(v)< \tilde{I}_A\cdot f(z)$ on $A$ by the definition of $A$, namely $\tilde{I}_A\cdot f(v)< \tilde{I}_A\cdot f(z)$ on $A$, which contradicts with $\tilde{I}_A\cdot f(v)=\tilde{I}_A\cdot f(z)$, and thus $(4)$  must hold.
\par
As for $(5)$, on one hand, $\tilde{I}_A\cdot (\alpha\cdot d(z,\bar{v})+f(\bar{v}))= \alpha \cdot \tilde{I}_A\cdot d(z,\bar{v})+ \tilde{I}_A\cdot f(\bar{v})=\tilde{I}_A\cdot \alpha \cdot d(z,v) +\tilde{I}_A\cdot f(v)$ $($ by noticing $\tilde{I}_A\cdot d(z,\bar{v})= \tilde{I}_A\cdot d(z,v)$ by the triangle inequality $)<\tilde{I}_A\cdot f(z)$ on $A$, namely $\alpha\cdot d(z,\bar{v})+ f(\bar{v})< f(z)$ on $A$; on the other hand, $\tilde{I}_{A^c}\cdot (\alpha\cdot d(z,\bar{v})+f(\bar{v}))= \alpha \cdot \tilde{I}_{A^c}\cdot d(z,\bar{v})+ \tilde{I}_{A^c}\cdot f(\bar{v})=\tilde{I}_{A^c}\cdot f(z)$ $($ by noticing $\tilde{I}_{A^c}\cdot d(z,\bar{v})=0$ by the definition of $\bar{v} )$, so $\alpha\cdot d(z,\bar{v})+f(\bar{v})=f(z)$ on $A^c$. To sum up, $\alpha\cdot d(z,\bar{v})+f(\bar{v}) \leq f(z)$, namely $(5)$ also holds, but this contradicts $(3)'$.
\end{proof}

\begin{theorem}\label{theorem3.7}
Let $(E,d)$ be a $d$--$\sigma$--stable $L^0$--complete $RM$ space with base $(\Omega,\mathcal{F},P)$, $\varepsilon\in L^0_{++}(\mathcal{F})$ and $f: E\rightarrow \bar{L}^0(\mathcal{F})$ proper, $\sigma$--stable, $\mathcal{T}_c$--lower semicontinuous and bounded from below. Then for any given point $x_0$ in $E$ satisfying $f(x_0)\leq \bigwedge f(E)+\varepsilon$ and any given $\alpha\in L^0_{++}(\mathcal{F})$, there exists $z\in E$ such that the following hold:
\begin{enumerate}[$(1)$]
\item $f(z)\leq f(x_0)-\alpha\cdot d(z,x_0)$;
\item $d(z,x_0)\leq \alpha^{-1}\cdot \varepsilon$;
\item For each $x\in E$ such that $x\neq z, f(x)>f(z)-\alpha\cdot d(x,z)$.
\end{enumerate}
\end{theorem}

\begin{proof}
It follows from Theorems \ref{theorem2.13}, \ref{theorem3.5} and \ref{theorem3.6}
\end{proof}
\par
For the $(\varepsilon,\lambda)$--version of the Caristi's fixed point theorem in complete $RM$ spaces, please refer to \cite[Theorem 2.14]{GY12}, which, combined with Theorem \ref{theorem2.13} and \ref{theorem3.5}, leads directly to the following:

\begin{theorem}\label{theorem3.8}
Let $(E,d)$ be a $d$--$\sigma$--stable $L^0$--complete $RM$ space with base $(\Omega,\mathcal{F},P)$ and $f: E\rightarrow \bar{L}^0(\mathcal{F})$ proper, $\sigma$--stable, $\mathcal{T}_c$--lower semicontinuous and bounded from below. If $T: E \rightarrow E$ satisfies $f(T(x))+ d(T(x),x)\leq f(x)$ for any $x\in E$, then $T$ has a fixed point.
\end{theorem}
\par
Let $(E,d)$ be an $RM$ space with base $(\Omega,\mathcal{F},P)$, a nonempty subset $G$ of $E$ is said to a.s. bounded if $D(G): =\bigvee\{d(x,y): x,y\in G\}\in L^0_+(\mathcal{F})$, called the random diameter of $D$. In fact, $G$ is a.s. bounded iff $\{d(x,y): x,y\in G\}$ is bounded in order in $(L^0(\mathcal{F}),\leq)$.

\begin{definition}\label{definition3.9}
Let $(E,d)$ be an $RM$ space with base $(\Omega,\mathcal{F},P)$, $CB(E)$ the family of a.s. bounded and $\mathcal{T}_{\varepsilon,\lambda}$--closed nonempty subsets of $E$ and $CB_{\sigma}(E)=\{G\in CB(E): G$ is $d$--$\sigma$--stable$\}$. Define the random Hausdorff metric $H: CB_{\sigma}(E)\times CB_{\sigma}(E) \rightarrow L^0_+(\mathcal{F})$ by $H(G_1,G_2)=\max\{\bigvee_{x_1\in G_1}d(x_1,G_2), \bigvee_{x_2\in G_2} d(x_2,G_1)\}$ for any $G_1$ and $G_2$ in $CB_{\sigma}(E)$, where $d(x,G)=\bigwedge \{d(x,g): g\in G\}$ denotes the random distance from $x\in E$ to a nonempty subset $G$ of $E$.
\end{definition}

\begin{remark}\label{remark3.10}
By $(2)$ of Theorem \ref{theorem2.12}, it is easy to check that $(CB_{\sigma}(E), H)$ is an $RM$ space with base $(\Omega,\mathcal{F},P)$, and if $(E,d)$ is an $(\varepsilon,\lambda)$--complete $RM$ space then $(CB_{\sigma}(E), H)$ is also $(\varepsilon,\lambda)$--complete by a similar reasoning of the classical Hausdorff distance.
\end{remark}

\begin{lemma}\label{lemma3.11}
Let $(E,d)$ be the same as in Definition \ref{definition3.9}, $\varepsilon\in L^0_{++}(\mathcal{F})$, $G_1$ and $G_2\in CB_{\sigma}(E)$. Then for any given $g_1\in G_1$, there exists some $g_2\in G_2$ such that $d(g_1,g_2)\leq d(g_1,G_2)+\varepsilon$.
\end{lemma}

\begin{proof}
By $(4)$ of Theorem \ref{theorem2.10}, $\{d(g_1,g): g\in G_2\}$ is $\sigma$--stable. Applying Lemma \ref{lemma3.1} to $\{d(g_1,g): g\in G_2\}$ yields some $g_2\in G_2$ satisfying our desire.
\end{proof}

\par
Theorem \ref{theorem3.12} below generalizes Nadler's fixed point theorem from a complete metric space to an $(\varepsilon,\lambda)$--complete $RM$ space.

\begin{theorem}\label{theorem3.12}
Let $(E,d)$ be an $(\varepsilon,\lambda)$--complete $RM$ space with base $(\Omega,\mathcal{F},P)$, $\alpha\in L^0_+(\mathcal{F})$ satisfying $\alpha<1$ on $\Omega$ and $T: E \rightarrow CB_{\sigma}(E)$ a mapping such that $H(T(x), T(y))\leq \alpha\cdot d(x,y)$ for any $x$ and $y\in E$. Then there exists $x\in E$ such that $x\in T(x)$.
\end{theorem}

\begin{proof}
Let $\alpha^0$ be an arbitrarily chosen representative of $\alpha$ and $A=\{\omega\in \Omega: \alpha^0(\omega)=0\}$, define $\hat{\alpha}^0(\omega)=\alpha^0(\omega)$ if $\alpha^0(\omega)>0$ and $\hat{\alpha}^0(\omega)=\frac{1}{2}$ if $\omega\in A$. Further, let $\hat{\alpha}$ be the equivalence class of $\hat{\alpha}^0$, then $\hat{\alpha}\in L^0_{++}(\mathcal{F})$ and satisfies $H(T(x),T(y))\leq \hat{\alpha}\cdot d(x,y)$ for any $x,y\in E$. Thus, we can, without loss of generality, assume $\alpha\in L^0_{++}(\mathcal{F})$.
\par
Taking a given point $x_0\in E$ and $x_1\in T(x_0)$, then there exists some $x_2\in T(x_1)$ by Lemma \ref{lemma3.11} such that $d(x_1,x_2)\leq d(x_1,T(x_1))+ \alpha\leq H(T(x_0), T(x_1))+ \alpha$. Again by Lemma \ref{lemma3.11} there exists $x_3\in T(x_2)$ such that $d(x_2,x_3)\leq d(x_2,T(x_2))+\alpha^2\leq H(T(x_1), T(x_2))+\alpha^2 $.
\par
By induction, there exists a sequence $\{x_n: n\in N\}$ in $E$ such that $x_n\in T(x_{n-1})$ and $d(x_n,x_{n+1})\leq H(T(x_{n-1}), T(x_n))+ \alpha^n$ for any $n\geq 1$. Then it is easy to obtain that $d(x_n,x_{n+1})\leq \alpha^n d(x_0,x_1)+ n\alpha^n$ for any $n\geq 1$. Thus for any $n\leq m, d(x_n,x_{m+1})\leq \sum^m_{i=n}d(x_i,x_{i+1})\leq \sum^m_{i=n} \alpha^i\cdot d(x_0,x_1)+ \sum^m_{i=n}i\cdot \alpha^i$. Further, since $\alpha\in L^0_{++}(\mathcal{F})$ and $\alpha< 1$ on $\Omega$, $d(x_n,x_m)$ converges a.s. to $0$ when $n,m$ tend to $+\infty$, $\{x_n: n\in N\}$ is, of course, a Cauchy sequence in $E$ under the $(\varepsilon,\lambda)$--uniformity on $E$, and hence convergent to some $x\in E$. It follows that for any $n\in N$, $d(x,T(x))\leq d(x,x_n)+ d(x_n,T(x))\leq d(x,x_n)+ \alpha\cdot d(x_{n-1},x)$, so $d(x,T(x))=0$, namely $x\in T(x)$.
\end{proof}

Although the shape and idea of proof of Theorem \ref{theorem3.12} are the same as those of the classical Nadler's fixed point theorem of \cite{N69}, Theorem \ref{theorem3.12} contains more as attested by the following series of corollaries of it mainly because we employ the framework of an $RM$ space!
\par
Let us recall some basic concepts on measurable multivalued functions and multivalued mappings: let $(X,d)$ be a metric space and $2^X$ the family of subsets of $X$, a multivalued function $V: (\Omega,\mathcal{F},P) \rightarrow 2^X$ is said to be measurable $($or weakly measurable in terms of \cite{I77,H75}$)$ if $V^{-1}(G):=\{\omega\in \Omega: V(\omega)\cap G\neq \emptyset\}\in \mathcal{F}$ for any $d$--open set $G$. A mapping $T: \Omega\times X\rightarrow 2^X$ is said to be a multivalued random operator if $T(\cdot,x): \Omega\rightarrow 2^X$ is measurable for each $x\in X$. For the study of random fixed points of multivalued random operators, see, e.g.\cite{BAX66,FMM09,I77,I79,S01}. Corollary \ref{corollary3.13} below can be regarded as a generalization of a random fixed point theorem due to S.Iton\cite{I77}. In fact, we give a new proof of \cite[Theorem]{I77} .

\begin{corollary}\label{corollary3.13}
Let $(X,d)$ be a polish space, $\alpha^0: \Omega\rightarrow [0,+\infty)$ a random variable on $(\Omega,\mathcal{F},P)$ such that $0\leq \alpha^0<1$ a.s. and $T: \Omega\times X\rightarrow CB(X)$ a multivalued random operator, where $CB(X)$ is the family of nonempty bounded closed subsets of $X$. If the following conditions are satisfied:
\begin{enumerate}[$(1)$]
\item There exists $\Omega_0\in \mathcal{F}$ with $P(\Omega_0)=1$ such that $T(\omega,\cdot): (X,d)\rightarrow (CB(X),h)$ is continuous for each $\omega\in \Omega_0$, where $h$ denotes the classical Hausdorff metric on $CB(X)$;
\item $P(\Omega(x,y))=1$ for any $(x,y)\in X \times X$, where $\Omega(x,y)=\{\omega\in \Omega: h(T(\omega,x), T(\omega,y))\leq \alpha^0(\omega)d(x,y)\}$ is assumed to be $\mathcal{F}$--measurable.
\end{enumerate}
Then there exists a random element $x^0: (\Omega,\mathcal{F}, P)\rightarrow (X,d)$ such that $x^0(\omega)\in T(\omega,x^0(\omega))$ for almost all $\omega\in \Omega$.
\end{corollary}

\begin{proof}
Let $\{x_n:n\in N\}$ is a countable dense subset of $X$ and $\Omega_1=\Omega'\cap \Omega_0\cap (\cap_{i,j}\Omega(x_i,x_j))$, then $\Omega_1\in \mathcal{F}$ and $P(\Omega_1)=1$, where $\Omega'=\{\omega\in \Omega: 0\leq \alpha^0(\omega)<1\}$. Further, by $(1)$ and $(2)$ one can have: $h(T(\omega,x),T(\omega,y))\leq \alpha^0(\omega)d(x,y)$ for any $\omega\in \Omega_1$ and any $x,y\in X$. We can, without loss of generality, assume $\Omega_1=\Omega$ $($ since otherwise we can consider $(\Omega_1,\mathcal{F}_1, P_1)$ instead of $(\Omega,\mathcal{F},P)$, where $\mathcal{F}_1=\Omega_1\cap \mathcal{F}$ and $P_1=P|_{\mathcal{F}_1} )$
\par
Thus, for any two random elements $x^0$ and $y^0: \Omega\rightarrow X$, one has that $h(T(\omega,x^0(\omega)), T(\omega,y^0(\omega)))\leq \alpha^0(\omega)d(x^0(\omega), y^0(\omega))$ for each $\omega\in \Omega$. By Proposition 2 of \cite{I77}, $F: \Omega\rightarrow CB(X)$ defined by $F(\omega)=T(\omega,x^0(\omega))$ for each $\omega\in \Omega$ and each random element $x^0: \Omega \rightarrow X$, is measurable, which induces a mapping $\hat{T}: L^0(\mathcal{F},X)\rightarrow CB_{\sigma}(L^0(\mathcal{F},X))$ by $\hat{T}(g)=\{\hat{g}\in L^0(\mathcal{F},X): \hat{g}$ is the equivalence class of some measurable selection of $F=T(\cdot,g^0(\cdot))\}$, where $g$ is the equivalence class of the random element $g^0: \Omega\rightarrow X$. By Example \ref{example2.8} and \ref{example2.9}, $\hat{T}$ is well defined.
\par
Let $g^0_1$ and $g^0_2$ be two random elements: $\Omega\rightarrow X$ and $g_1$ and $g_2$ respectively their equivalence classes. By Theorem 5.6 of \cite{H75}, there exist two sequences $\{u^0_n: n\in N\}$ and $\{v^0_n: n\in N\}$ of random elements such that $T(\omega,g^0_1(\omega))= cl(\{u^0_n(\omega): n\in N\})$ and $T(\omega,g^0_2(\omega))=cl(\{v^0_n(\omega): n\in N\})$ for each $\omega\in \Omega$, where $``cl"$ stands for the $d$--closure operation. Then $h(T(\omega,g^0_1(\omega)), T(\omega,g^0_2(\omega)))= \max\{\sup_{i\geq 1}\inf_{j\geq 1}d(u^0_i(\omega),\\ v^0_j(\omega)), \sup_{j\geq 1}\inf_{i\geq 1}d(u^0_i(\omega), v^0_j(\omega))\}$, see \cite[pp.88]{I77}. It is not very difficult to check that $H(\hat{T}(g_1),\hat{T}(g_2))= \max \{\bigvee_{i\geq 1}\bigwedge_{j\geq 1}d(u_i,v_j), \bigvee_{j\geq 1}\bigwedge_{i\geq 1}d(u_i,v_j) \}$, where $u_i$ and $v_j$ are the equivalence classes of $u^0_i$ and $v^0_j$, respectively. Thus $H(\hat{T}(g_1),\hat{T}(g_2))\leq \alpha\cdot d(g_1,g_2)$, where $\alpha$ is the equivalence class of $\alpha^0$. By Theorem\ref{theorem3.12}, there exists some $x\in L^0(\mathcal{F}, X)$ such that $x\in \hat{T}(x)$, then an arbitrarily chosen representative $x^0$ of $x$ must satisfy $x^0(\omega)\in T(\omega,x^0(\omega))$ for almost all $\omega\in \Omega$.
\end{proof}

\begin{corollary}\label{corollary3.14}\cite{Guo99}
Let $(E,d)$ be an $(\varepsilon,\lambda)$--complete $RM$ space with base $(\Omega,\mathcal{F},P), \alpha\in L^0_+(\mathcal{F})$ such that $\alpha< 1$ on $\Omega$, and $T: E\rightarrow E$ a mapping satisfying $d(T(x), T(y))\leq \alpha\cdot d(x,y)$ for any $x,y\in E$. Then there exists unique one $x\in E$ such that $T(x)=x$.
\end{corollary}

\begin{proof}
Define a multivalued mapping $\hat{T}: E \rightarrow CB_{\sigma}(E)$ by $\hat{T}(x)=\{T(x)\}$ for each $x\in E$. Since each singleton set $\{T(x)\}\in CB_{\sigma}(E), \hat{T}$ is well defined and $H(\hat{T}(x), \hat{T}(y))= d(T(x), T(y))\leq \alpha\cdot d(x,y)$ for any $x,y\in E$. It follows from Theorem \ref{theorem3.12} that there exists $x\in E$ such that $x\in \hat{T}(x)= \{T(x)\},$ namely $T(x)=x$. The uniqueness of $x$ comes from the random contraction condition.
\end{proof}

\begin{corollary}\label{corollary3.15}
Let $(E,d)$ be a $d$--$\sigma$--stable $L^0$--complete $RM$ space with base $(\Omega,\mathcal{F},P), \alpha\in L^0_+(\mathcal{F})$ satisfying $\alpha<1$ on $\Omega$, and $T: E\rightarrow E$ satisfying $d(T(x), T(y))\leq \alpha\cdot d(x,y)$ for any $x,y\in E$. Then there exists unique one $x\in E$ such that $T(x)=x$.
\end{corollary}

\begin{proof}
It follows from Theorem \ref{theorem2.13} and Corollary \ref{corollary3.14}.
\end{proof}

\begin{remark}\label{remark3.16}
Corollary \ref{corollary3.14} also has a slightly general formulation: if there exists some $n\in N$ such that $d(T^n(x), T^n(y))\leq \alpha\cdot d(x,y)$ for any $x,y\in E$, then $T$ still has unique one $x\in E$ such that $T(x)=x$. Proof is very familiar as follows: by Corollary \ref{corollary3.14}, $T^n$, denoting the $n$th iterate of $T$, has unique one fixed point $x\in E$, since $T^n$ and $T$ are commutative, namely $T^n\circ T= T\circ T^n$, then $T^n(T(x))= T(T^n(x))=T(x)$, one has $T(x)=x$. But it is very the simple observation that motivates Theorem \ref{theorem3.17} below.
\end{remark}

\par
Let $(E,d)$ be a $d$--$\sigma$--stable $RM$ space with base $(\Omega,\mathcal{F},P)$, $T: E\rightarrow E$  and $L: \Omega\rightarrow N$ a positive integer--valued random variable. Define $T^L: E\rightarrow E$ by $T^L(x)=\sum_{k=1}^{\infty}\tilde{I}_{(L=k)}\cdot T^k(x)$ for any $x\in E$, where $(L=k)=\{\omega\in \Omega: L(\omega)=k\}$. $T$ is said to be $\sigma$--stable if $T(\sum_{n=1}^{\infty}\tilde{I}_{A_n}\cdot x_n)=\sum_{n=1}^{\infty}\tilde{I}_{A_n}\cdot T(x_n) $ for each sequence $\{x_n: n\in N\}$ in $E$ and each countable partition $\{A_n: n\in N\}$ of $\Omega$ to $\mathcal{F}$. It is obvious that $T$ and $T^L$ are commutative when $T$ is $\sigma$--stable since $T^L(T(x))=\sum_{k=1}^{\infty}\tilde{I}_{(L=k)}\cdot T^k(T(x))=\sum_{k=1}^{\infty}\tilde{I}_{(L=k)}\cdot T(T^k(x))=T(\sum_{k=1}^{\infty}\tilde{I}_{(L=k)}\cdot T^k(x))=T(T^L(x)) $ for any $x\in E$.

\begin{theorem}\label{theorem3.17}
Let $(E,d)$ be a $d$--$\sigma$--stable $(\varepsilon,\lambda)$--complete $RM$ space with base $(\Omega,\mathcal{F},P)$, $T: E\rightarrow E$ a $\sigma$--stable mapping, $L: \Omega\rightarrow N$ a random variable and $\alpha\in L^0_+(\mathcal{F})$ such that $\alpha<1$ on $\Omega$ and $d(T^L(x),T^L(y))\leq \alpha\cdot d(x,y)$ for any $x,y\in E$. Then $T$ has unique one fixed point.
\end{theorem}

\begin{proof}
It follows from Corollary \ref{corollary3.14} that $T^L$ has unique one fixed point $x\in E$. Since $T$ and $T^L$ are commutative, $x$ is also the unique fixed point of $T$.
\end{proof}

\begin{corollary}\label{corollary3.18}\cite{GZWG18}
Let $(E,\|\cdot\|)$ be a $\mathcal{T}_{\varepsilon,\lambda}$--complete $RN$ module over $K$ with base $(\Omega,\mathcal{F},P)$, $G$ a $\sigma$--stable $\mathcal{T}_{\varepsilon,\lambda}$--closed subset of $E$, $T:G\rightarrow G$ a $\sigma$--stable mapping, $L: \Omega\rightarrow N$ a random variable and $\alpha\in L^0_+(\mathcal{F})$ such that $\alpha<1$ on $\Omega$ and $\|T^L(x)-T^L(y)\|\leq \alpha\cdot \|x-y\|$ for any $x,y\in G$. Then $T$ has unique one fixed point $x\in G$.
\end{corollary}

\begin{proof}
Since $(G,d)$ is also a $d$--$\sigma$--stable $(\varepsilon,\lambda)$--complete $RM$ space with base $(\Omega,\mathcal{F},P)$, where $d(g_1,g_2)=\|g_1-g_2\|$ for any $g_1,g_2\in G$. Then applying Theorem \ref{theorem3.17} to $(G,d)$ ends the proof.
\end{proof}

In the final part of this paper, let us derive the two random fixed point theorems due to Han\v{s}\cite{H61}, which are the earliest random fixed point theorems in probabilistic functional analysis initiated by A.\v{S}pac\v{e}k and O.Han\v{s}.

\begin{corollary}\label{corollary3.19}\cite{H61}
Let $(X,d)$ be a polish space and $T: \Omega\times X\rightarrow X$ a continuous random operator $($namely $T(\omega,\cdot)$ is continuous for each $\omega\in \Omega )$ such that the following condition holds:\\
$(1)$. $P(\cup^{\infty}_{m=1} \cup^{\infty}_{n=1} \cap_{x\in X} \cap_{y\in X} \{\omega\in \Omega: d(T^n(\omega,x), T^n(\omega,y))\leq (1-\frac{1}{m})d(x,y)\})=1$\\
Then there exists an $X$--valued random element $x^0$ such that $T(\omega,x^0(\omega))=x^0(\omega)$ for almost all $\omega$ in $\Omega$ and $x^0$ is unique a.s.. Here, $T^n(\omega,x)=T(\omega,T^{n-1}(\omega,x))$ and $T^0(\omega,x)=x$ for each $(\omega,x)\in \Omega\times X$ and each $n\geq 1$.
\end{corollary}

\begin{proof}
Let $\Omega_1=\cup^{\infty}_{m=1} \cup^{\infty}_{n=1} \cap_{x\in X} \cap_{y\in X} \{\omega\in \Omega: d(T^n(\omega,x), T^n(\omega,y))\leq (1-\frac{1}{m})d(x,y)\}$ and $\Omega_2=\cup^{\infty}_{m=1} \cup^{\infty}_{n=1} \cap_{i=1}^{\infty} \cap_{j=1}^{\infty} \{\omega\in \Omega: d(T^n(\omega,x_i), T^n(\omega,x_j))\leq (1-\frac{1}{m})d(x_i,x_j)\}$, where $\{x_i: i\geq 1\}$ is a countable dense subset of $X$. Then, according to continuity of $T$, $\Omega_1=\Omega_2$ is $\mathcal{F}$--measurable. We can, without loss of generality, assume $\Omega_1=\Omega_2=\Omega$. For each $m\geq 1$, define $B_m=\cup^{\infty}_{n=1} \cap_{i=1}^{\infty} \cap_{j=1}^{\infty}\{\omega\in \Omega: d(T^n(\omega,x_i), T^n(\omega,x_j))\leq (1-\frac{1}{m}) d(x_i,x_j)\}$, then $\{B_m,\: m\geq 1\}$ is a nondecreasing sequence in $\mathcal{F}$ and $\cup_{m=1}^{\infty} B_m=\Omega$, further, let $A_1=B_1$ and $A_n=B_n\setminus A_{n-1}$ for any $n\geq 2$, then $\{A_n: n\in N\}$ forms a countable partition of $\Omega$ to $\mathcal{F}$.
\par
Now, define a positive--integer--valued random variable $L: \Omega\rightarrow N$ by $L(\omega)=\min\{n\geq 1: d(T^n(\omega,x_i), T^n(\omega,x_j))\leq (1-\frac{1}{m})\cdot d(x_i,x_j)$ for any $i$ and $j$ in $N\}$ when $\omega\in A_m$ for some $m\in N$. Then it is easy to check that $L$ is well defined, it is also obvious that $T^L: \Omega\times X\rightarrow X$ defined by $T^L(\omega,x)=T^{L(\omega)}(\omega,x)=T^k(\omega,x)$ when $\omega\in (L=k)$ for any $(\omega,x)\in  \Omega\times X$, is still a continuous random operator.
\par
Again, define a nonnegative random variable $\alpha^0: \Omega\rightarrow [0,1)$  by $\alpha^0(\omega)=1-\frac{1}{m}$ when $\omega\in A_m$ for some $m\in N$. Then, it is easy to see that $d(T^L(\omega,x_i), T^L(\omega,x_j))\leq \alpha^0(\omega)\cdot d(x_i,x_j)$ for each $\omega\in \Omega$ and each $(i,j)\in N\times N$, so that $d(T^L(\omega,x), T^L(\omega,y))\leq \alpha^0(\omega)\cdot d(x,y)$ for each $\omega\in \Omega$ and each $(x,y)\in X\times X$. Further, one also has $d(T^L(\omega,x^0(\omega)), T^L(\omega,y^0(\omega)))\leq \alpha^0(\omega)\cdot d(x^0(\omega), y^0(\omega))$ for each $\omega\in \Omega$ and any two $X$--valued random elements $x^0$ and $y^0$.
\par
Now, let $(L^0(\mathcal{F}, X), d)$ be as in Example \ref{example2.8}, which is a $d$--$\sigma$--stable $(\varepsilon,\lambda)$--complete $RM$ space with base $(\Omega,\mathcal{F},P)$ and define $\hat{T}: L^0(\mathcal{F}, X)\rightarrow L^0(\mathcal{F}, X)$ by $\hat{T}(x)=$ the equivalence class of $T(\cdot, x^0(\cdot))$ for any $x\in L^0(\mathcal{F}, X)$, where $x^0$ is an arbitrarily chosen representative of $x$ and $T(\cdot, x^0(\cdot))(\omega)= T(\omega, x^0(\omega))$ for each $\omega\in \Omega$, then it is easy to check that $\hat{T}^L(x)$ is just the equivalence class of $T^L(\cdot, x^0(\cdot))$, where $x$ and $x^0$ are the same as in the definition of $\hat{T}$. Thus $d(\hat{T}^L(x), \hat{T}^L(y))\leq \alpha\cdot d(x,y)$ for any $x,y\in L^0(\mathcal{F}, X)$. By noticing $\hat{T}$ is obviously $\sigma$--stable, it follows from Theorem \ref{theorem3.17} that $\hat{T}$  has unique one fixed point $x\in L^0(\mathcal{F}, X)$, and an arbitrarily chosen representative $x^0$ of $x$ satisfies our needs.
\end{proof}

\begin{corollary}\label{corollary3.20}
Let $(X,d)$ be a complete metric space, $T: \Omega\times X\rightarrow X$ a continuous strong random operator $($where ``strong" means $T(\cdot,x)$ is an $X$--valued strong random element for each $x\in X$, see Example \ref{example2.8} for the concept of strong random elements$)$ and $\alpha^0: \Omega\rightarrow [0,+\infty)$ a random variable satisfying $\alpha^0<1$ a.s. . If $d(T(\omega,x), T(\omega,y))\leq \alpha^0(\omega)\cdot d(x,y)$ for almost all $\omega$ in $\Omega$ and any given $x,y\in X$, namely $P\{\omega\in \Omega: d(T(\omega,x), T(\omega,y))\leq \alpha^0(\omega)\cdot d(x,y)\}=1$ for any $(x,y)\in X\times X$, then there exists an $X$--valued strong random element $x^0$ such that $T(\omega,x^0(\omega))=x^0(\omega)$ for almost all $\omega$ in $\Omega$, and $x^0$ is unique a.s. .
\end{corollary}

\begin{proof}
Similarly to the proof of Corollary \ref{corollary3.19}, define $\hat{T}: L^0(\mathcal{F}, X) \rightarrow L^0(\mathcal{F}, X)$ by $\hat{T}(x)=$ the equivalence class of $T(\cdot, x^0(\cdot))$ for any $x\in L^0(\mathcal{F}, X)$, where $x^0$ is an arbitrarily chosen representative of $x$. Since $x^0$ has a separable range and $T$ is a continuous strong random operator, $\hat{T}$ is well defined and $d(\hat{T}(x),\hat{T}(y))\leq \alpha\cdot d(x,y)$ for any $x,y\in L^0(\mathcal{F}, X)$, where $\alpha$ is the equivalence class of $\alpha^0$. Since $L^0(\mathcal{F}, X)$ is $(\varepsilon,\lambda)$--complete, it follows from Corollary\ref{corollary3.14} that $\hat{T}$ has unique one fixed point $x\in L^0(\mathcal{F}, X)$, whose arbitrarily chosen representative $x^0$ satisfies our needs.
\end{proof}

\begin{remark}\label{remark3.21}
When $(X,d)$ is a polish space and $T$ is a continuous random operator, Corollary \ref{corollary3.20} is due to O.Han\v{s} $($see \cite{BR76}$)$. The formulation of our Corollary \ref{corollary3.20} has an advantage that the separability of $X$ can be removed. This advantage continues to be reflected in our study of random fixed point theorems for random nonexpansive operators \cite{GZWY19}, please compare \cite{GZWY19} with \cite{Xu90}. Probabilistic functional analysis initiated by A.\v{S}pac\v{e}k and O.Han\v{s} $($\cite{S55,H61}$)$ are concerned with theories of random elements and random operators. When regarding random elements as points in $RM$ spaces or $RN$ modules or random locally convex modules, and correspondingly regarding random operators as mappings between $RM$ spaces or $RN$ modules and et.al., probabilistic functional analysis can be  naturally regarded as a part of random functional analysis, which is just the idea of developing probabilistic functional analysis in \cite{Guo93}. Now, random functional analysis based on $RM$ spaces, $RN$ modules and random locally convex modules, has undergone a systematic and deep development in the direction of traditional functional analysis, connected with this is that probabilistic functional analysis also has obtained a corresponding development. We may hope that the approach to probabilistic functional analysis will develop a greater power in the further.
\end{remark}


\begin{thebibliography}{99}

\bibitem{B22}
S.Banach, {\it Sur les operations dans les ensembles abstraits et leur applications aux equations int\'{e}grales}, Fund.Math., {\bf3}(1922),133--181.

\bibitem{BAX66}
T.D.Benavides, G.L.Acedo and H.K.Xu, {\it Random fixed points of set--valued operators}, Proc.Amer.Math.Soc., {\bf124}(1996), 831--838.

\bibitem{BR76}
A.T.Bharucha--Reid, {\it Fixed point theorems in probabilistic ananlysis}, Bull.Amer.Math.Soc., {\bf82(5)}(1976), 641--657.

\bibitem{CKV15}
P.Cheridito, M.Kupper and N.Vogelpoth, {\it Conditional analysis on $R^{d}$, Set Optimization and Applications}, Proceedings in Mathematics \& Statistics, {\bf 151}(2015), 179--211.

\bibitem{DS06}
F.Delbaen and W.Schachermayer, {\it The Mathematics of Arbitrage}, Springer--Verlag, Berlin, 2006.

\bibitem{DJKK16}
S.Drapeau, A.Jamneshan, M.Karliczek and M.Kupper, {\it The algebra of conditional sets and the concepts of conditional topology and compactness}, J.Math,Anal.Appl., {\bf 437(1)}(2016), 561--589.

\bibitem{DKKS13}
S.Drapeau, M.Karliczek, M,Kupper and M.Streckfuss, {\it Brouwer fixed point theorem in $(L^0)^d$}, Fixed Point Theory Appl., {\bf 301(1)}(2013).

\bibitem{DS58}
N.Dunford and J.T.Schwartz, {\it Linear Operators$(I)$: General Theory}, John Wilely \& Sons Inc., New York, 1958.

\bibitem{Eke74}
I.Ekeland, {\it On the variational principle}, J. Math. Anal. Appl., {\bf 47} (1974), 324--353.

\bibitem{FMM09}
R.Fierro, C.Martinez and C.H.Morales, {\it Fixed point theorems for random lower semicontinuous mappings}, Fixed Point Theory Appl., {\bf 2009}(2009), ID584178.

\bibitem{FKV09}
D.Filipovi\'{c}, M.Kupper and N.Vogelpoth, {\it Separation and duality in locally $L^{0}-$convex modules}, J.Funct.Anal., {\bf 256} (2009), 3996--4029.

\bibitem{FM14a}
M.Frittelli and M.Maggis, {\it Complete duality for quasiconvex dynamic risk measures on modules of $L^p$--type}, Statist. \& Risk Model, {\bf 31(1)} (2014), 103--128.

\bibitem{FM14b}
M.Frittelli and M.Maggis, {\it Conditionally evenly convex sets and evenly quasi-convex maps}, J.Math.Anal.Appl., {\bf 413} (2014), 169--184.

\bibitem{Guo89}T.X.Guo,{\it The theory of probabilistic metric spaces with applications to random functional analysis}, Master's thesis, Xi'an Jiaotong University (China), 1989.

\bibitem{Guo92}
T.X.Guo, {\it Random metric theory and its applications}, Ph.D thesis, Xi'an Jiaotong University (China), 1992.

\bibitem{Guo93}
T.X.Guo, {\it A new approach to random functional analysis}, in: Proceedings of the first China postdoctral academic conference, The China National Defense and Industry Press, Beijing, 1993, pp.1150--1154.

\bibitem{Guo96}
T.X.Guo, {\it The Radon--Nikod\'{y}m property of conjugate spaces and the $w^*$--equivalence theorem for $w^*$--measurable functions}, Sci. China Math. Ser.A, {\bf 39}(1996), 1034--1041.

\bibitem{Guo96a}
T.X.Guo, {\it Module homomorphisms on random normed modules}, Chinese Northeastern Math.J., {\bf 12}(1996), 102--114.

\bibitem{Guo99}
T.X.Guo, {\it Some basic theories of random normed linear spaces and random inner product spaces}, Acta Anal.Funct.Appl., {\bf 1(2)}(1999), 160--184.

\bibitem{Guo08}
T.X.Guo, {\it The relation of Banach--Alaoglu theorem and Banach--Bourbaki--Kakutani--\v{S}mulian theorem in complete random normed modules to stratification structure}, Sci.China Math.Ser.A, {\bf 51(9)}(2008), 1651--1663.

\bibitem{Guo10}
T.X.Guo, {\it Relations between some basic results derived from two kinds of topologies for a random locally convex module}, J.Funct.Anal., {\bf 258} (2010), 3024--3047.

\bibitem{Guo13}
T.X.Guo, {\it On some basic theorems of continuous module homomorphisms between random normed modules}, J.Funct. Spaces Appl., (2013), Article ID 989102, 13 pages.

\bibitem{Guo09}
T.X.Guo and X.X.Chen, {\it Random duality}, Sci.China Math. Ser.A, {\bf 52}(2009), 2084--2098.

\bibitem{GL05}
T.X.Guo and S.B.Li. {\it The James theorem in complete random normed modules}, J. Math. Anal. Appl.,{\bf 308}(2005), 257--265.

\bibitem{GP01}
T.X.Guo and S.L.Peng, {\it A characterization for an $L(\mu,K)$--topological module to admit enough canonical module homomorphisms}, J.Math.Anal.Appl., {\bf 263}(2001), 580--599.

\bibitem{GS11}
T.X.Guo and G.Shi. {\it The algebraic structure of finitely generated $L^0(\mathcal{F}, K)$--modules and the Helly theorem in random normed modules}, J. Math. Anal. Appl., {\bf 381}(2011), 833--842.

\bibitem{GXC09}
T.X.Guo, H.X.Xiao and X.X.Chen, {\it A basic strict separation theorem in random locally convex modules}, Nonlinear Anal., {\bf 71} (2009), 3794--3804.

\bibitem{GY12}
T.X.Guo and Y.J.Yang, {\it Ekeland's variational principle for an $\bar{L}^0-$valued function on a complete random metric space}, J. Math. Anal. Appl., {\bf 389} (2012), 1--14.

\bibitem{GY96}
T.X.Guo and Z.Y.You. {\it The Riesz's representation theorem in complete random inner product modules and its applications}, Chin. Ann. Math. Ser. A, {\bf 17}(1996), 361--364.

\bibitem{GZWG18}
T.X.Guo, E.X.Zhang, Y.C.Wang and Z.C.Guo. {\it Two fixed point theorems in complete random normed modules and their applications to backward stochastic equations}, 2018, arXiv: 1801.09341v3.

\bibitem{GZWY19}
T.X.Guo, E.X.Zhang, Y.C.Wang and G.Yuan, {\it $L^0$--convex compactness and random normal structure in $L^0(\mathcal{F}, B)$}, 2019, arXiv:1904.03607.

\bibitem{GZWYYZ17}
T.X.Guo, E.X.Zhang, M.Z.Wu, B.X.Yang, G.Yuan and X.L.Zeng, {\it On random convex analysis}, J.Nonlinear Conv.Anal., {\bf 18(11)}(2017), 1967--1996.

\bibitem{GZZ14}
T.X.Guo, S.E.Zhao and X.L.Zeng, {\it The relations among the three kinds of conditional risk measures}, Sci. China Math. {\bf 57(8)}(2014), 1753--1764.

\bibitem{H61}
O.Han\v{s}, {\it Random operator equations}, Proceedings of the $4th$ Berkely symposium on mathematical statistics and probability, Vol.\uppercase\expandafter{\romannumeral2}, Part\uppercase\expandafter{\romannumeral1}, pp.185--202, Univ. of California Press, Berkeley, 1961.

\bibitem{HR87}
L.P.Hansen and S.F.Richard, {\it The role of conditioning information in deducing testable restrictions implied by dynamic asset pricing models},Econometrica, {\bf 55(3)}(1987), 587--613.

\bibitem{HLR91}
R.Haydon, M.Levy and R.Raynaud, {\it Randomly Normed Spaces}, Hermann, Paris,1991.

\bibitem{H75}
C.J.Himmelberg, {\it Measurable relations}, Fund.Math., {\bf 87}(1975), 53--72.

\bibitem{I77}
S.Iton, {\it A random fixed point theorem for a multivalued contraction mapping}, Pacific J.Math., {\bf 68}(1977), 85--90

\bibitem{I79}
S.Iton, {\it Random fixed point theorems with an application to random differential equations in Banach spaces}, J.Math.Anal.Appl., {\bf 67}(1979), 261--273.

\bibitem{JKZ18}
A.Jamneshan, M.Kupper and J.M.Zapata, {\it Parameter--dependent stochastic optimal control in finite discrete time}, 2018, arXiv: 1705.02374v2.

\bibitem{M05}
I.Molchanov, {\it Theory of random sets}, Probability and Its Applications, Springer, Lodon, 2005.

\bibitem{N69}
S.B.Nadler, Jr., {\it Multivalued contraction mappings}, Pacific J.Math., {\bf 30}(1969), 475--488.

\bibitem{SS8305}
B.Schweizer and A.Sklar, {\it Probabilistic Metric Spaces}, Elservier, New York,1983; reissued by Dover Publications, New York, 2005.

\bibitem{S01}
N.Shahzad, {\it Random fixed points of set--valued maps}, Nonlinear Anal., {\bf 45}(2001), 689--692.

\bibitem{S55}
A.\v{S}pac\v{e}k, {\it Zuf\"{a}llige Gleichungen}, Czechoslovak Math.J., {\bf 5}(1955), 462--466.

\bibitem{S56}
A.\v{S}pac\v{e}k, {\it Note on K.Menger's probabilistic geometry}, Czechoslovak Math.J., {\bf 6(81)}(1956), 72--74.

\bibitem{S60}
A.\v{S}pac\v{e}k, {\it Random metric spaces}, Trans. Second Prague conf. inform. theory, statist. decision functions and random processes, 1960, pp. 627--638.

\bibitem{Wu12}
M.Z.Wu, {\it The Bishop-Phelps theorem in complete random normed modules endowed with the $(\varepsilon,\lambda)$-topology}, J.Math.Anal.Appl., {\bf 391}(2012), 648--652.

\bibitem{Wu13}
M.Z.Wu, {\it Farkas' lemma in random locally convex modules and Minkowski--Weyl type results in $L^0(\mathcal{F},R^n)$},J.Math.Anal.Appl., {\bf 404}(2013), 300--309.

\bibitem{WG15}
M.Z.Wu and T.X.Guo, {\it A counterexample shows that not every locally $L^0$--convex topology is necessarily induced by a family of $L^0$--seminorms}, arXiv:1501.04400V1,2015.

\bibitem{Xu90}
H.K.Xu, {\it Some random fixed point theorems for condensing and nonexpansive operators}, Proc. Amer. Math. Soc., {\bf 110}(1990), 395--400.

\bibitem{YZG91}
Z.Y.You, L.H.Zhu and T.X.Guo, {\it Random conjugate spaces for a class of quasinormed linear spaces}, J.Xi'an Jiaotong Univ., {\bf 3}(1991), 133--134(Abstract in Chinese).

\bibitem{Z17}
J.M.Zapata, {\it On the characterization of locally $L^0$--convex topologies induced by a family of $L^0$--seminorms}, J.Conv.Anal., {\bf 24(2)}(2017), 383--391.

\end{thebibliography}
\end{document}